\documentclass[onefignum,onetabnum]{siamart251216}

\headers{Method for QPE}
{X. Han, K. Jiang, M. Li}

\title{Accurately computing quasiperiodic parabolic equations within finite-size domains via modeling quasiperiodic boundary conditions
}

\author{
Xiaofang Han\thanks{
Hunan Key Laboratory for Computation and Simulation in Science and Engineering,
Key Laboratory of Intelligent Computing and Information Processing of Ministry of Education,
School of Mathematics and Computational Science,
Xiangtan University, Xiangtan, Hunan, 411105, China
(\email{xiaofang\_han@smail.xtu.edu.cn},
\email{limeng@smail.xtu.edu.cn}).
}
\and
Kai Jiang\thanks{
Wuhan Institute for Math \& AI,
Wuhan University, Wuhan, Hubei, 430072, China
(\email{kaijiang@whu.edu.cn}).
}
\and
Meng Li\footnotemark[1]
}
\usepackage{mathrsfs,amsmath,amssymb,bm}
\usepackage{lipsum}
\usepackage{amsfonts}
\usepackage{graphicx, subfigure}
\usepackage{epstopdf}
\usepackage{makecell, rotating, bbding}
\usepackage{multirow}
\usepackage{algorithmic, algorithm}
\usepackage{enumerate}
\usepackage{booktabs}
\usepackage[misc]{ifsym}
\usepackage{mathtools}
\usepackage[draft]{changes}
\usepackage{amsopn}
\usepackage{aligned-overset}
\usepackage{nicematrix}
\usepackage{tikz}
\usepackage{dashbox}
\usetikzlibrary{calc}
\usepackage{hyperref}
\usepackage{enumitem}
\usepackage{cases}

\newcommand{\bx}{\bm{x}}

\newtheorem{thm}{Theorem}[section]

\newtheorem{remark}[thm]{Remark}

\usepackage[draft]{changes}
\usepackage{lipsum}

\newsiamremark{hypothesis}{Hypothesis}
\crefname{hypothesis}{Hypothesis}{Hypotheses}
\newsiamthm{claim}{Claim}

\usepackage{amsopn}

\makeatletter
\newcommand*{\addFileDependency}[1]{
  \typeout{(#1)}
  \@addtofilelist{#1}
  \IfFileExists{#1}{}{\typeout{No file #1.}}
}
\makeatother

\usepackage{xr}

\begin{document}

\maketitle


\begin{abstract}
Quasiperiodic parabolic equations (QPEs) describe various physical processes in systems with long-range order without decay, such as heat conduction, diffusion, and transport, and are naturally posed on the whole space.
However, practical computations could be performed on finite domains, where preserving the global quasiperiodic structure through appropriate boundary conditions remains a key challenge. Traditional periodic boundary conditions (PBCs) cannot capture quasiperiodic information beyond the computational domain and suffer from Diophantine errors caused by rational approximations of irrational numbers. To address this issue, we propose a class of quasiperiodic boundary conditions (QBCs) based on a homomorphism between the physical domain and a high-dimensional torus, enabling the long-range quasiperiodic structure to be captured at the boundaries. We further establish rigorous convergence results for numerical discretizations of finite-size QPEs with QBCs. Numerical experiments demonstrate that QBCs avoid the influence of Diophantine errors and enable accurate and efficient computations for both high- and low-regularity cases, with improved convergence over PBCs.

\end{abstract}



\section{Introduction}\label{sec:int} 
Quasiperiodic parabolic equations provide a natural model for physical processes in media with long-range order but no translational invariance.
Typical applications include heat conduction in quasicrystals, diffusion in
porous media, and transport phenomena in fluids
\cite{archambault1997thermal,gao2023effects,maestrello1979quasiperiodic,mathieu2016method,takeda1999quasiperiodic}.
Such problems are conventionally formulated on the whole space. 
The corresponding linear QPEs take the form
\begin{equation}\label{eq:intro_whole_qpe}
\begin{cases}
\dfrac{\partial u(\bm{x},t)}{\partial t}
-\mathrm{div}\bigl(\alpha(\bm{x})\nabla u(\bm{x},t)\bigr)
=f(\bm{x},t),
& (\bm{x},t)\in\mathbb{R}^d\times(0,T),\\[2mm]
u(\bm{x},0)=u_0(\bm{x}),
& \bm{x}\in\mathbb{R}^d.
\end{cases}
\end{equation}
Here, the coefficient $\alpha(\bm{x})$ is quasiperiodic
(see \Cref{def:QP_fun2}) and satisfies 
\[
\gamma_1 |\bm{\xi}|^2
\leq
\alpha(\bm{x}) |\bm{\xi}|^2
\leq
\gamma_2 |\bm{\xi}|^2,
\qquad
\bm{x}, 
\bm{\xi}\in\mathbb{R}^d,
\]
where $0<\gamma_1\leq\gamma_2$. 
For quasiperiodic coefficients, source terms, and initial values, the parabolic equation
\eqref{eq:intro_whole_qpe} admits quasiperiodic solutions under suitable assumptions
\cite{bagirov1975stabilization,Blanc2010Improving,Jiang2025Convergence}. 
In practical applications, however, quasiperiodic materials are usually 
confined to finite spatial domains, which motivates the formulation of
QPEs on bounded domains together with suitable
boundary conditions. 
Accordingly, we consider the linear QPE on a finite-size domain 
\(\Omega\times(0,T)\subseteq\mathbb{R}^d\times\mathbb{R}\), hereafter referred
to as the finite-size QPE
\begin{equation}\label{eq:intro_qpe_model}
\begin{cases}
\dfrac{\partial u(\bm{x},t)}{\partial t}
-\mathrm{div}\bigl(\alpha(\bm{x})\nabla u(\bm{x},t)\bigr)
=f(\bm{x},t),
& (\bm{x},t)\in\Omega\times(0,T),\\[2mm]
u(\bm{x},0)=u_0(\bm{x}),
& \bm{x}\in\Omega,\\[1mm]
u(\bm{x},t)\ \text{satisfies a suitable boundary condition},
& (\bm{x},t)\in\partial\Omega\times(0,T).
\end{cases}
\end{equation}
QPEs naturally arise on the whole space, while finite-size modeling requires appropriate boundary conditions. Such boundary conditions should account for the
quasiperiodic information from the external region while preserving the quasiperiodic
structure within the domain. The lack of translational invariance and decay in
quasiperiodic solutions makes the construction of such boundary conditions challenging. 
A widely used approach is periodic boundary condition (PBC), 
which imposes
\[
u(\bm{x} + \bm{L}, t) = u(\bm{x}, t), \quad (\bm{x},t) \in \partial \Omega \times (0, T), \quad \bm{L} = (L_j)_{j=1}^d.
\]
\begin{figure}[!htbp]
    \centering
    \subfigure[]{
        \includegraphics[width=4cm]{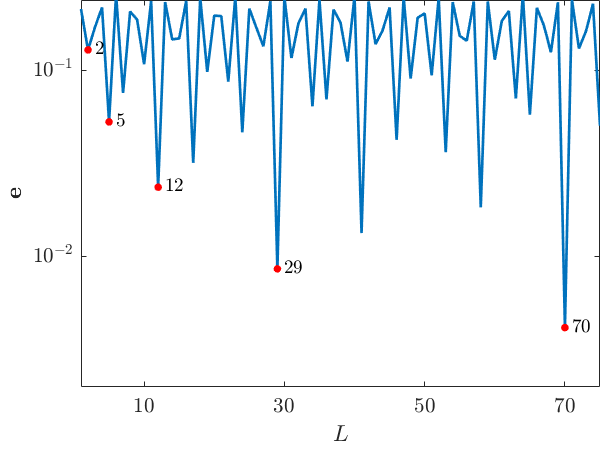}
        \label{fig:pbc_1}
    }
    \hspace{2mm}
    \subfigure[]{
        \includegraphics[width=4cm]{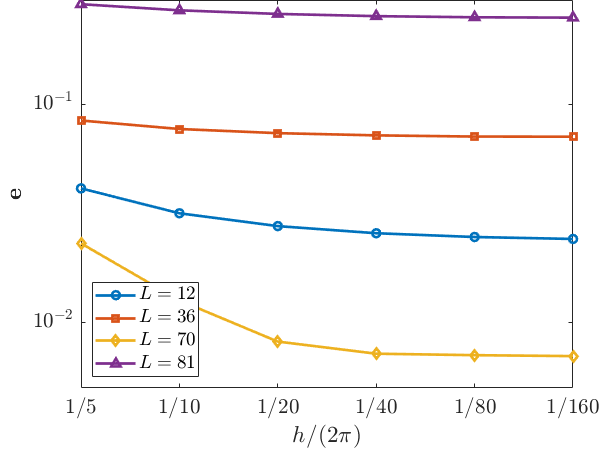}
        \label{fig:pbc_2}
    }
    \caption{
PBC errors of the QPE \eqref{eq:intro_qpe_model} with $\alpha(x)=\cos x+\cos \sqrt{2}x+6$. \Cref{fig:pbc_1} shows the errors against $L$ with $h=\pi/1280$, and \Cref{fig:pbc_2} presents the $\ell^\infty$-norm errors with decreasing step $h$ for fixed $L$. }
\label{fig:pbc}
\end{figure}
For the finite-size QPE \eqref{eq:intro_qpe_model}, imposing PBCs leads to a periodic approximation of the corresponding whole-space quasiperiodic solution. Since the quasiperiodic structure involves irrational frequencies, such a periodic approximation necessarily requires rational approximations of the irrational frequencies, thereby introducing the Diophantine error. 
For an irrational number $\beta$, the corresponding Diophantine error is defined as  
\[
e_L=|L\beta-[L\beta]|,\qquad L\in\mathbb{Z}^+,
\]
where $[\cdot]$ denotes the nearest integer.
The continued-fraction
expansion of $\beta$ is
\[
\beta=[a_0;a_1,a_2,\ldots]
=a_0+\cfrac{1}{a_1+\cfrac{1}{a_2+\ddots}},
\]
where $a_0\in\mathbb{Z}$ and $a_\nu\in\mathbb{N}$ for $\nu\geq1$ are the
partial quotients. 
The finite truncation
$[a_0;a_1,\ldots,a_\nu]$ is the $\nu$-th convergent of $\beta$, and its
denominator $L_\nu$ is called the $\nu$-th convergent denominator of $\beta$. 
The influence of the Diophantine error is also observed when PBCs are used to solve QPEs
\eqref{eq:intro_qpe_model} on $\Omega=[-L\pi,L\pi]$ with 
$
\alpha(x)=\cos x+\cos(\beta x)+6,\  \beta=\sqrt{2}.
$ 
The numerical error is measured in the $\ell^\infty$-norm at the final time, and the fully discrete scheme with PBCs is given in \Cref{app:PBC}. 
\Cref{fig:pbc_sr5_2} shows that $e_L$ varies nonmonotonically with $L$ and decreases consistently only when $L$ takes the convergent denominators $L_\nu$. 
To reduce the spatial discretization error, we take $h=\pi/1280$, as shown in \Cref{fig:pbc_1}. 
The PBC error shows a similar dependence on \(L\), decreasing consistently only at the convergent denominators. 
In \Cref{fig:pbc_2}, for each fixed $L$, the PBC error reaches
a plateau as $h$ decreases, while smaller errors are generally obtained
when $L$ is chosen as a convergent denominator of $\beta$.
These results indicate that the PBC error in solving QPEs \eqref{eq:intro_qpe_model} is dominated by the Diophantine error and converges extremely slowly.

Recently, the projection method (PM) \cite{Jiang2014Numericalpm,Jiang2018Numerical} has emerged as an accurate and efficient approach for solving quasiperiodic systems  with relatively high regularity on the whole space. By embedding a quasiperiodic problem into a higher-dimensional periodic problem, the PM enables  computation in the periodic parent space. It has been applied to quasicrystal computations, quasiperiodic homogenization, incommensurate quantum systems, and interfacial problems
\cite{Chen2026structure,Jiang2014Numericalpm,Jiang2018Numerical,Jiang2025Convergence,jiang2025projection,Jiang2024Numerical,Jiang2024IWFPM,Jiang2024tilt,Quan2026spectral,Xiong2026frame}. 
However, the PM is  designed for whole-space quasiperiodic systems
and may not be directly suited to finite-size problems. Commonly used boundary conditions, such as PBCs, are generally affected by Diophantine errors. 
Therefore, new boundary conditions are required to preserve the quasiperiodic structure in finite-size computations.

In this work, we propose quasiperiodic boundary conditions (QBCs) for finite-size QPEs by exploiting the homomorphic relationship between the physical space and the irrational manifold on high-dimensional torus \cite{Fan2025representation}. The QBCs preserve the quasiperiodic structure of the whole-space problem and capture the long-range order at the computational boundaries, with further applications to quasiperiodic Hamilton--Jacobi equations and related homogenization problems \cite{Jiang2026HJ}. For the numerical realization of QBCs, we employ the finite point recovery (FPR) method \cite{jiang2024accurately}. 
Through local interpolation of a finite set of interior samples,
the FPR method is used solely to recover quasiperiodic function values directly in the physical space, 
thereby avoiding dimensional lifting in numerical computation. The function values recovered by the FPR method approximate the boundary values specified by the QBCs. Higher-order FPR interpolation can further reduce the boundary approximation error.
Compared with  PBCs, QBCs avoid the impact of the Diophantine error and exhibit enhanced convergence.
Moreover, local basis functions are constructed from local information and used to recover the boundary information required by the QBCs,
making the method particularly suitable for
low-regularity cases. For arbitrary \(p\geq2,~q\geq 1\), we discretize the diffusion term and the time
derivative using the \(p\)-th order central difference formula (CDF-\(p\))
and the \(q\)-th order backward differentiation formula (BDF-\(q\)),
respectively. We then establish an error analysis for the resulting fully
discrete scheme. Finally, extensive numerical experiments are conducted for both high- and low-regularity finite-size QPEs. 
The results not only validate the theoretical analysis but also demonstrate that the QBCs achieve high accuracy and improved convergence behavior compared to  PBCs, while exhibiting enhanced computational efficiency in low-regularity cases relative to the PM.

The article is organized as follows. \Cref{sec:preliminaries} presents a brief overview of the quasiperiodic functions and the FPR method.   \Cref{sec:QBC} introduces the concept of the QBCs and gives their numerical implementation. 
Moreover, we discretize \eqref{eq:intro_qpe_model} using CDF-$p$ schemes in space and BDF-$q$ schemes in time, and derive the resulting linear systems.
\Cref{sec:anal} establishes the error analysis for the fully discrete formulation. 
\Cref{sec:num} shows numerical experiments using the second- and
fourth-order CDF and BDF schemes for spatial and temporal discretizations,
respectively, to verify the convergence analysis. The results for both
high- and low-regularity cases demonstrate the accuracy of the proposed
QBCs for the finite-size QPE \eqref{eq:intro_qpe_model}.
Finally, \Cref{sec:conc} summarizes our findings and presents future perspectives.

\section{Preliminaries}\label{sec:preliminaries} 
This section provides the preliminaries on quasiperiodic functions and the FPR method, laying the  foundation for the subsequent analysis and numerical solution of \eqref{eq:intro_qpe_model}.
\subsection{Quasiperiodic functions} 
We first give the definition of quasiperiodic functions and some of their necessary properties.

\begin{definition}
A matrix $\bm P\in\mathbb R^{d\times n}$ is called a projection matrix if
\[
\bm P\in\mathbb P^{d\times n}
:=
\left\{
\bm P=(\bm p_1,\ldots,\bm p_n)\in\mathbb R^{d\times n}:
\operatorname{rank}_{\mathbb Q}(\bm p_1,\ldots,\bm p_n)=n
\right\}.
\]
\end{definition}

\begin{definition}\label{def:QP_fun2}
A continuous function $f(\bm{x}):\mathbb{R}^d\to\mathbb{R}$ is
quasiperiodic if there exist an $n$-dimensional periodic function
$F\in C(\mathbb{T}^n)$ and a projection matrix
$\bm{P}\in\mathbb{P}^{d\times n}$ such that
\begin{equation}
f(\bm{x})=F(\Phi_{\bm{x}}^{\bm{P}}),
\qquad
\Phi_{\bm{x}}^{\bm{P}}
:=
(\bm{P}^T\bm{x})/(2\pi\mathbb{Z}^n) \in \mathbb{T}^n.
\label{eq:qp_def}
\end{equation}
Here, $F$ is called the parent function of $f$, and
$\Phi_{\bm{x}}^{\bm{P}}$ denotes the orbit point associated with
$\bm{x}$. The map
$
\bm{x}\mapsto\Phi_{\bm{x}}^{\bm{P}}
$
is the physical-to-torus homomorphism \cite{Fan2025representation},
whose image forms the associated irrational manifold on $\mathbb{T}^n$.
The set $\mathrm{QP}(\mathbb{R}^d)$ consists of all $d$-dimensional
quasiperiodic functions.
\end{definition}

Let $K_T
=
\left\{
\bm{x}=(x_1,\ldots,x_d)\in\mathbb{R}^d:
|x_j|\leq T,\ j=1,\ldots,d
\right\}$.
For \(f\in\mathrm{QP}(\mathbb{R}^d)\), its mean value is defined by
\begin{equation}\label{eqn:elliptic:meanvalue}
\mathcal{M}\{f\}
:=
\lim_{T\to\infty}
\frac{1}{(2T)^d}
\int_{\bm{s}+K_T} f(\bm{x})\,d\bm{x},
\end{equation}
where the limit exists uniformly with respect to
\(\bm{s}\in\mathbb{R}^d\) \cite{Jiang2024Numerical}.


\begin{definition}
Let $f\in {\rm QP}(\mathbb{R}^d)$. For
$\bm{\lambda}\in\mathbb{R}^d$, the Bohr--Fourier coefficient
of $f$ is defined by
\[
\widehat{f}(\bm{\lambda})
:=
\mathcal{M}\left\{
f(\bm{x})e^{-\imath\bm{\lambda}^{T}\bm x}
\right\}.
\]
The Bohr spectrum of $f$ is defined as
\[
\Lambda(f)
:=
\left\{
\bm{\lambda}\in\mathbb{R}^d:
\widehat{f}(\bm{\lambda})\neq0
\right\}.
\]
For $\rho\geq0$, define
\[
|f|_{\mathcal{H}_{\rm QP}^{\rho}(\mathbb{R}^d)}^{2}
:=
\sum_{\bm{\lambda}\in\Lambda(f)\setminus\{\bm{0}\}}
|\bm{\lambda}|^{2\rho}
|\widehat{f}(\bm{\lambda})|^{2}.
\]
We write
$f\in\mathcal{H}_{\rm QP}^{\rho}(\mathbb{R}^d)$
if this seminorm is finite.
\end{definition}

Then, we have the following property of quasiperiodic functions.

\begin{lemma}[\cite{Fan2025representation}, Theorem 1.1]
\label{lem:ergodic}
Let $\bm P\in\mathbb P^{d\times n}$. 
Consider the dynamical system
$\bigl(\mathbb T^n,(\Phi_{\bm x}^{\bm P})_{\bm x\in\mathbb R^d}\bigr)$
associated with the projection matrix $\bm P$.
Then the system is uniquely ergodic and minimal.
Consequently, the irrational manifold
$\{\Phi_{\bm x}^{\bm P}:\bm x\in\mathbb R^d\}$
is dense in $\mathbb T^n$, i.e.,
\[
\overline{
\{\Phi_{\bm x}^{\bm P}:\bm x\in\mathbb R^d\}}
=
\mathbb T^n.
\]
\end{lemma}

\subsection{FPR method}Recently, 
the FPR method  has been developed to recover the global high- and low-regularity quasiperiodic functions \cite{jiang2024accurately}. The idea is to select nearby interpolation nodes in the high-dimensional torus  and use their function values to recover the target value. 
This idea is implemented through the following procedure.
\begin{enumerate}
\item Choose a domain $\Omega\subset\mathbb{R}^d$ and a physical grid
$\Omega_N\subset\Omega$ such that the orbit points
$\{\Phi_{\bm x_i}^{\bm P}:\bm x_i\in\Omega_N\}$ sufficiently sample $\mathbb{T}^n$.

\item For a target point $\bm x^*\in\mathbb R^d$, compute its orbit point 
$\bm\theta^*=\Phi_{\bm x^*}^{\bm P}$ and select $K$ nearby orbit points,
denoted by
$
    \Theta_K
    =
    \left\{
    \Phi_{\bm {x}_{i_j}}^{\bm P}
    \right\}_{j=1}^{K}.
$
	
	\item Construct the local interpolation basis functions
	$\{\phi_j\}_{j=1}^{K}$ using the phases in $\Theta_K$.
	
   \item
   Recover the function value at $\bm{x}^{*}$ by
    \begin{equation*}
        u(\bm{x}^{*})
        =
        U\left(\Phi_{\bm{x}^{*}}^{\bm P}\right)
        \approx
        \mathcal{T}_{\mathrm{FPR}(k)}u(\bm{x}^{*})
        :=
        \sum_{j=1}^{K}
        u\left(\bm{x}_{i_j}\right)
        \phi_j\left(\Phi_{\bm{x}^{*}}^{\bm P}\right).
    \end{equation*}


\end{enumerate}


When $k$-degree  Lagrange interpolation is employed, the FPR interpolation operator is denoted by $\mathcal{T}_{\mathrm{FPR}(k)}$. The following lemma provides its error estimate. 
\begin{lemma}
\label{lem:FPR_error1}
Let $u\in\mathcal{H}_{\rm QP}^{\rho}(\mathbb{R}^{d})$ with
$\rho>k+1+n/2$. Assume that there exists a constant
$c_{\bm P}>0$ such that
\begin{equation}
|\bm P\bm\zeta|
\geq c_{\bm P}|\bm\zeta|,
\qquad
\bm\zeta\in\mathbb{Z}^{n}\setminus\{\bm 0\},
\quad
\bm P\bm\zeta\in\Lambda(u).
\label{eq:spec_separa}
\end{equation}
Then, the FPR-$k$ interpolation satisfies
\[
\|u-\mathcal{T}_{\mathrm{FPR}(k)}u\|_{\infty}
\leq
C h^{k+1}
|u|_{\mathcal{H}_{\rm QP}^{\rho}(\mathbb{R}^{d})},
\]
where $C$ is independent of $h$ but depends on $c_{\bm P}$.
\end{lemma}
A complete proof of this result is provided in
\cite[Theorem~3.2]{Jiang2026HJ}, based on the regularity estimate associated with
\cite[Theorem~6.5]{Fan2025representation} and the FPR interpolation
error estimate in \cite[Theorem~4.4]{jiang2024accurately}.

\section{Computational framework for finite-size QPEs}\label{sec:QBC}
In this section, we provide computational modeling for finite-size QPEs. Specifically, we introduce QBCs for \eqref{eq:intro_qpe_model} and implement them using the FPR method.


\subsection{Quasiperiodic boundary conditions} 
The basic idea of QBCs is to use the mapping $\Phi_{\bm x}^{\bm P}$
to associate boundary points in the physical domain with their corresponding
orbit points on the high-dimensional torus, and then perform the quasiperiodic boundary recovery by the FPR method. 
Consider a simple example illustrated in \Cref{fig:omgeatour}.
A line segment represents the 1D physical domain, and the red points mark
the boundary points. 
For each $x\in\Omega$, $\Phi_x^{\bm P}$ gives the corresponding phase
on the two-dimensional torus $\mathbb T^2$. The red orbit points correspond
to the boundary points of $\Omega$.

\begin{figure}[!htbp]
    \centering
    \subfigure[]{
        \includegraphics[width=3.5cm]{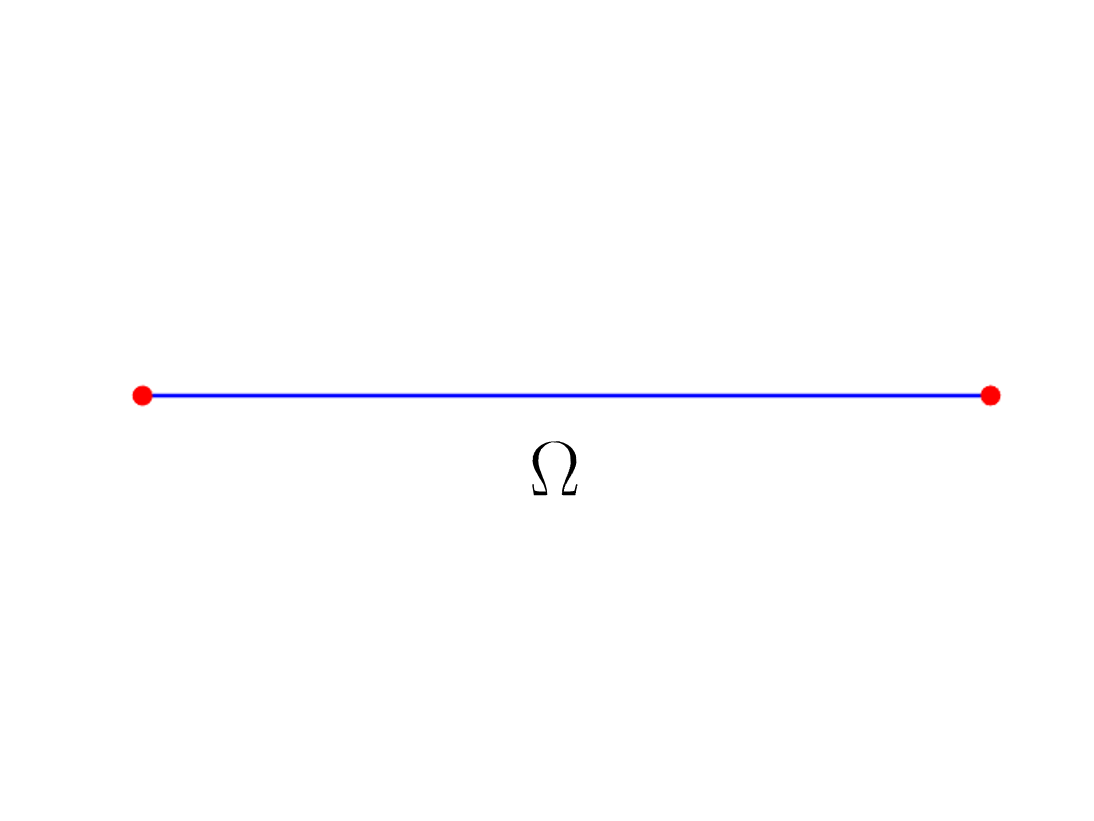}
    }
    \hspace{1.5cm}
    \subfigure[]{
        \includegraphics[width=4.5cm]{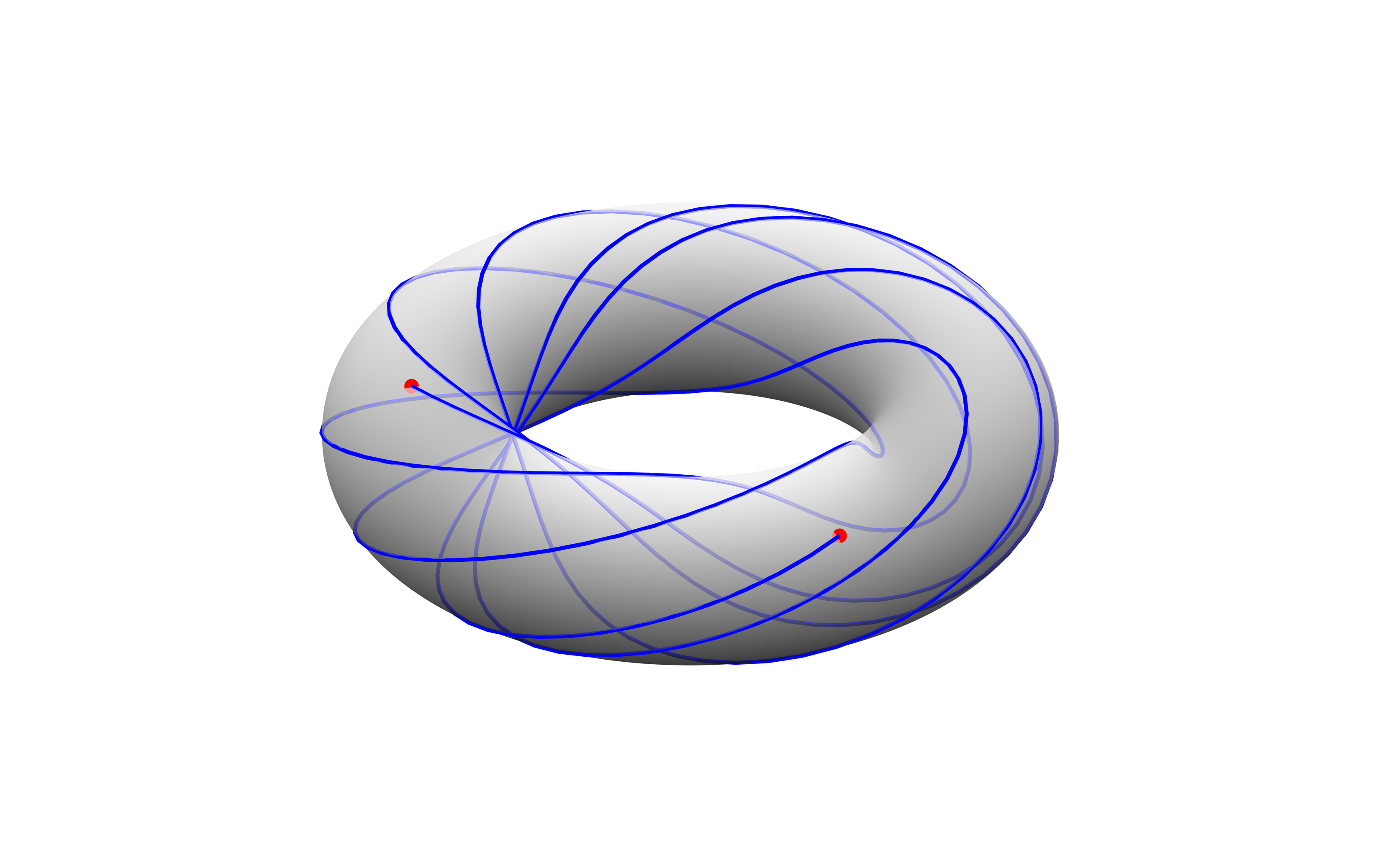}
    }
    \caption{
    Mapping of a 1D physical domain to a 2D irrational manifold,
    (a) the physical domain $\Omega$;
    (b) the corresponding irrational manifold
    $\{\Phi_x^{\bm P}:x\in\Omega\}$ on $\mathbb{T}^2$.}
    \label{fig:omgeatour}
\end{figure}


Based on this idea, we introduce three types of QBCs below. The first kind of QBC (QBC-I) is a Dirichlet-type condition, which defines the value of the function on the boundary. It is expressed as
\begin{equation*}\label{eq:QBC_def}
    u(\bm{x}) = U\left(\Phi_{\bm x}^{\bm P}\right), ~~~\bm{x} \in \partial \Omega.
\end{equation*}
Here, $u$ is the quasiperiodic function in the physical space
$\mathbb{R}^d$, $U$ is its periodic parent function on $\mathbb{T}^n$,
and $\bm P\in\mathbb{P}^{d\times n}$ is the associated projection matrix. 
The QBCs also admit a natural physical interpretation. To illustrate this, consider a quasiperiodic heat-conduction problem.  QBC-I determines the boundary temperature by restricting the full-space quasiperiodic temperature field to the computational boundary, rather than imposing an independent boundary value.


The second kind of QBC (QBC-II) is a Neumann-type condition, which specifies the normal derivative of the function on the boundary. It can be written as
	\begin{equation*}
		\partial_{\bm n}u(\bm{x})
		=
		\bm n\cdot
		\bm P\nabla_{\bm y}U
		\left(\Phi_{\bm{x}}^{\bm{P}}\right),
		\qquad
		\bm{x}\in\partial\Omega,
		\label{eq:QBC-II}
	\end{equation*}
where $\bm{n}$ denotes the outward unit normal on $\partial\Omega$, and $$ \bm{P}\cdot\nabla_{\bm y}=\sum_{i=1}^n\bm{p}_i\frac{\partial}{\partial y_i}, ~~\bm{y} = (y_1\cdots,y_n)\in \mathbb{T}^n.$$ 
For QBC-II, the computational boundary retains the normal heat-transfer
behavior of the full-space quasiperiodic medium. The boundary heat flux
is determined by the directional derivative of the higher-dimensional
periodic function $U$ along the irrational direction associated with
$\bm{P}^{T}$.

The third kind of QBC (QBC-III) is a Robin-type condition, which is a more general type of boundary condition that specifies a linear combination of the unknown function and its normal derivative  on the boundary.
Specifically, given constants $\omega_1,~\omega_2>0$, it is defined by
  \begin{equation*}
         	\displaystyle \omega_1 u(\bm x) + \omega_2 \partial_{ \bm{n}} u(\bm x)
         	= (\omega_1 U + \omega_2	\bm n\cdot
		\bm P\nabla_{\bm y} U)(\Phi_{\bm x}^{\bm P}),~~~ \bm x \in \partial\Omega.
    \end{equation*}
For QBC-III, the computational boundary preserves the heat-exchange
relation between the boundary temperature and the normal heat flux
inherited from the full-space quasiperiodic medium, with this coupling
governed by the irrational manifold associated with the projection
matrix $\bm{P}$.

\subsection{Discretization} \label{sbusec:dis}
The QPE \eqref{eq:intro_qpe_model} is discretized in space by a CDF-$p$
scheme  and in time by a BDF-$q$ scheme. 
For simplicity, we present the discretization for the 1D case and use the second-order spatial scheme as a representative example.

\textbf{Spatial discretization.}
We define a mesh on the computational domain $\Omega=(a,b)$. Let the step size be $h=(b-a)/N$. 
Then the mesh grid is defined as
\begin{equation*}
    \Omega_{N}:=\{x_i:x_i=i h+a,~0\leq i\leq N\}.
\end{equation*}
We denote the numerical solution at $x_i$ as $u_{i}(t)\approx u(x_i,t)$. The diffusion term is then discretized as 
\begin{equation}
- \nabla\left( \alpha(x_i)\nabla u_i(t)\right) \approx- \frac{1}{h^2} \left[ \alpha_{i+\frac{1}{2}} (u_{i+1}(t) - u_i(t)) - \alpha_{i-\frac{1}{2}} (u_i(t) - u_{i-1}(t)) \right],
\end{equation}
where $\alpha_{i\pm \frac{1}{2}} = \alpha(x_{i}\pm \frac{h}{2})$.

After spatial discretization, the QPE \eqref{eq:intro_qpe_model} reduces to the following semi-discrete system.
In vector form, let $
\bm{u}_h(t) := [u_0(t), u_1(t), \dots, u_N(t)]^T,$ $\bm{f}_h(t) := [f_0(t),f_1(t), \\  \dots, f_N(t)]^T.$
Then the semi-discrete system can be written as
\begin{equation}\label{eq:semiode}
\frac{d\bm u_h(t)}{dt}
+
A\bm u_h(t)
=
\bm f_h(t),
\end{equation}
where $A$ is the discrete diffusion matrix. The explicit structure of $A$ will be presented in \cref{sec:matrixstructure}.
 
It is worth noting that the interior discretization is not restricted to the present scheme, and other spatial discretization methods may also be employed.

\textbf{Temporal discretization.} 
Let the termination time be $T$, and divide the interval $[0,T]$ uniformly into $M$ intervals with time step
$$
\tau = \frac{T}{M}, \qquad t_m = m\tau, \quad m=0,1,\dots,M.
$$
Denote the approximate solution and right-hand side at the $m$-th time level as
$$
\bm{u}_h^m \approx \bm{u}_h(t_m), \qquad \bm{f}_h^m = \bm{f}_h(t_m),
$$
where $\bm{u}_h^m$ represents the numerical solution at time $t_m$ with step size $\tau$.

The $q$-th step BDF method approximates the time derivative by
\begin{equation}\label{eq:bdfk_formula}
\frac{1}{\tau}\sum_{j=0}^{q} c_j \bm{u}_h^{m+1-j} \approx \frac{d\bm{u}_h(t)}{dt}(t^{m+1}),
\end{equation}
where $\{c_j\}_{j=0}^q$ are the weight coefficients of the BDF-$q$ scheme. 

Substituting \eqref{eq:bdfk_formula} into the semi-discrete system \eqref{eq:semiode} yields the fully discrete BDF-$q$ scheme
\begin{equation}\label{eq:bdfk_general}
\frac{1}{\tau}\sum_{j=0}^{q}c_j\,\bm{u}_h^{\,m+1-j}
+ A\bm{u}_h^{\,m+1}
= \bm{f}_h^{\,m+1}.
\end{equation}

\textbf{Approximation of QBCs.}
Define the 1D boundary of the mesh grid $\Omega_N$ as
\begin{equation*}\label{eq:i_boundary}
    \partial\Omega_N := \left\{ x_i \in \Omega_N : x_i \in \{a,b\} \right\}.
\end{equation*}

Then we use the FPR method to discretize  the QBCs. 
As shown in \Cref{fig:omgeatourdis}, the red points denote the boundary
nodes $x_i\in\partial\Omega_N$ in the physical domain and the
corresponding orbit points $\Phi_{x_i}^{\bm P}$ on the irrational manifold
$\{\Phi_{x}^{\bm P}:x\in\Omega\}$, while the orange points denote the interpolation
orbit points  in $\Theta_K$. 
\begin{figure}[!htbp]
    \centering
    \subfigure[]{
        \includegraphics[width=3.5cm]{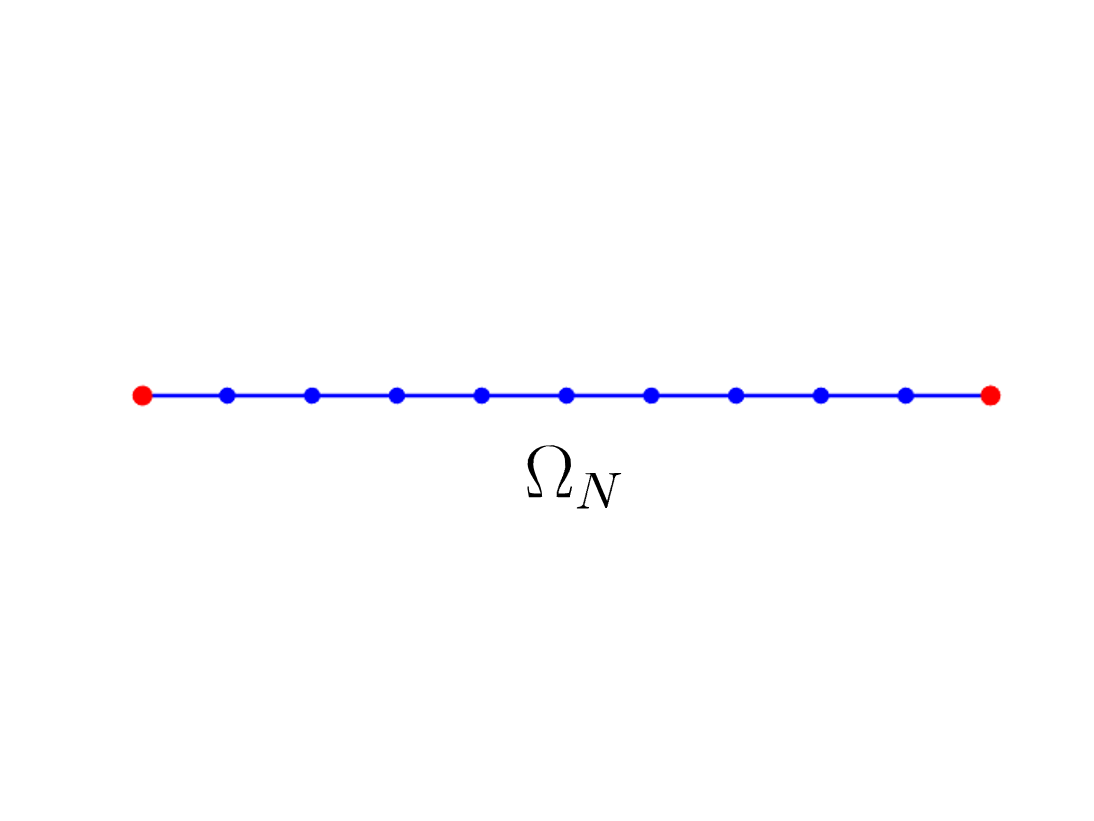}
    }
    \hspace{1.5cm}
    \subfigure[]{
        \includegraphics[width=4.5cm]{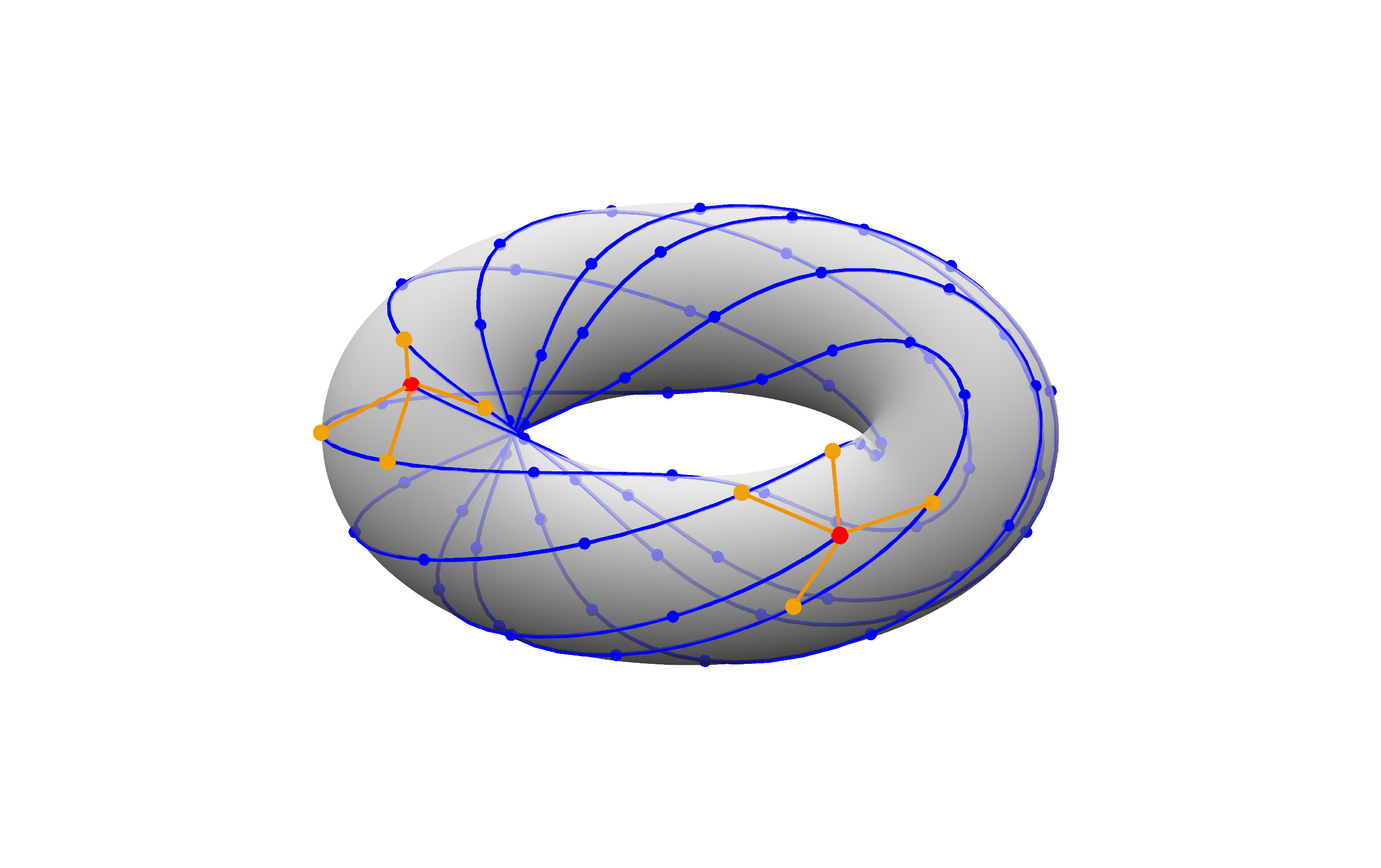}
    }
    \caption{
    Recovery of the boundary information for a 1D domain using the FPR method,
    (a) the physical grid $\Omega_N$;
    (b) the orbit points and interpolation set $\Theta_K$.}
    \label{fig:omgeatourdis}
\end{figure}

We directly select the mesh grid $\Omega_N$ as the pre-interpolation nodes to interpolate all the boundary points $x \in\partial\Omega_N$.
For each $x_i \in \Omega_N \setminus \Theta_K$, we have $\phi_i(\Phi_{ x}^{\bm P}) = 0$. 
The discrete QBC-I can be rewritten as
\begin{equation}\label{eq:discrete_QSEb_1D}
    u(x) =U(\Phi_{  x}^{\bm P})\approx \mathcal{T}_{\rm{ FPR}} u(x) = \sum_{x_i \in \Omega_N} u(x_i) \, \phi_i(\Phi_{ x}^{\bm P}), \qquad x \in \partial \Omega_N.
\end{equation}
We denote QBC-I with $k$-th order FPR interpolation by QBC-I(FPR$k$).

For QBC-II, the boundary derivative is recovered by the FPR method, 

 \begin{equation}
 \label{eq:qbc2_fpr}
 \frac{\partial u}{\partial x}\approx\mathcal{T}_{\mathrm{FPR}}\frac{\partial u}{\partial x}=\sum_{x_i\in\Omega_N}\left.\frac{\partial u}{\partial x}\right|_{x=x_i}\phi_i(\Phi_{  x}^{\bm P}), \quad x\in\partial\Omega_N.
 \end{equation}
To discretize QBC-II, the derivatives in \eqref{eq:qbc2_fpr} are approximated by a CDF scheme, with CDF-2 taken as a representative example. For any grid point $x_i\in\Omega_N$, the derivative is approximated by  
 \begin{equation}
\label{eq:qbc2_cdf}
\left.\frac{\partial u}{\partial x}\right|_{x=x_i}\approx\frac{u_{i+1}-u_{i-1}}{2h}.
 \end{equation}
 
For QBC-III, combining QBC-I and QBC-II gives
\begin{equation*}
\begin{aligned}
\omega_1 u( x) + \omega_2 \frac{\partial u}{\partial x}\approx \omega_1\sum_{x_i \in \Omega_N} u(x_i)\phi_i(\Phi_{  x}^{\bm P})+\omega_2\sum_{x_i\in\Omega_N}\left.\frac{\partial u}{\partial x}\right|_{x=x_i}\phi_i(\Phi_{  x}^{\bm P}),\qquad x \in \partial \Omega_N. 
\end{aligned}
\end{equation*}

\subsection{Matrix structure}\label{sec:matrixstructure}
To facilitate numerical implementation and analysis, we consider the fully discrete system   \eqref{eq:bdfk_general}, 
where the spatial discretization matrix $A$ is written as
\begin{equation}\label{eq:abd} 
A =- \frac{1}{h^2}(H + B),
\end{equation}
where $H, B \in \mathbb{R}^{(N+1)\times(N+1)}$ are specified  according to the imposed QBC. 
The matrix $H$ is obtained from the CDF-2  discretization of the diffusion term, while $B$ represents the boundary recovery matrix.
We specify the matrices $H$ and $B$ for the proposed QBCs.

\textbf{QBC-I.}  The matrix $H^{I}$ takes the tridiagonal form
\[
\scalebox{0.85}{$
H^{I} =
\begin{bmatrix}
0 & 0 & 0 & \cdots & \cdots & 0 \\
\alpha_{1/2} & -(\alpha_{1/2} + \alpha_{3/2}) & \alpha_{3/2} & 0 & \cdots & 0 \\
0 & \alpha_{3/2} & -(\alpha_{3/2} + \alpha_{5/2}) & \alpha_{5/2} & \cdots & 0  \\
\vdots & \ddots & \ddots & \ddots &\ddots & 0 \\
0 & \cdots & 0 & \alpha_{N-3/2} & -(\alpha_{N-3/2} + \alpha_{N-1/2}) & \alpha_{N-1/2} \\
0 & 0 & \cdots & \cdots & 0 & 0
\end{bmatrix}.
$}
\]
The corresponding boundary matrix $B^{I}$ is 
\begin{equation*}
    B^{I} =
    \begin{bmatrix}
        1&-\phi_{0,1}  & \cdots & -\phi_{0,N-1}&0 \\
        0 & \cdots&\cdots & \cdots&0 \\
        \vdots & \ddots & \vdots&\vdots&\vdots \\
        0&-\phi_{N,1}  & \cdots & -\phi_{N,N-1}&1
    \end{bmatrix}.
\end{equation*}

\textbf{QBC-II.}  Under QBC-II, the matrix $H^{II}$ is given by
\[
\scalebox{0.78}{$
\begingroup
\setlength{\arraycolsep}{2.5pt}
H^{II} =
\begin{bmatrix}
- (\alpha_{-1/2}+\alpha_{1/2}) & \alpha_{-1/2}+\alpha_{1/2}& 0 & \cdots &\cdots& 0 \\
\alpha_{1/2} & -(\alpha_{1/2} + \alpha_{3/2}) & \alpha_{3/2} & 0 & \cdots & 0 \\
\vdots & \ddots & \ddots & \ddots & \ddots & 0 \\
0 & \cdots & 0 & \alpha_{N-3/2} & -(\alpha_{N-3/2} + \alpha_{N-1/2}) & \alpha_{N-1/2} \\
0 & 0 & \cdots & 0 &\alpha_{N-1/2}+\alpha_{N+1/2} & -(\alpha_{N-1/2}+\alpha_{N+1/2})
\end{bmatrix}.
\endgroup
$}
\]
The boundary recovery matrix under QBC-II is written as
\[
B^{II} = B^{II}_1 + B^{II}_2,
\]
where $B^{II}_1$ and $B^{II}_2$ are defined as follows
\begin{equation*}
   B^{II}_1=
    \begin{bmatrix}
    \alpha_{-1/2}\phi_{0,1}&\cdots&\alpha_{-1/2}\phi_{0,N-1}&0&0\\
    0&0&\cdots&0&0\\
\vdots&\vdots&\ddots&\vdots&\vdots\\
      -\alpha_{N+1/2}\phi_{N,1}&\cdots&-\alpha_{N+1/2}\phi_{N,N-1}&0&0\\
    \end{bmatrix},
\end{equation*}
\begin{equation*}
    B^{II}_2=
    \begin{bmatrix}
    0&0&-\alpha_{-1/2}\phi_{0,1}&\cdots&-\alpha_{-1/2}\phi_{0,N-1}\\
    0&0&\cdots&0&0\\
\vdots&\vdots&\ddots&\vdots&\vdots\\
     0&0& \alpha_{N+1/2}\phi_{N,1}&\cdots&\alpha_{N+1/2}\phi_{N,N-1}\\
    \end{bmatrix}.
\end{equation*}

\begin{remark}
Since QBC-III involves a linear combination of the function value and its normal derivative, the corresponding boundary matrix is obtained by the same linear combination of the boundary matrices for QBC-I and QBC-II.
\end{remark}

\section{Error analysis}\label{sec:anal} In this section, we present a rigorous error analysis for combining CDF-$p$  and BDF-$q$ schemes  in solving the QPE \eqref{eq:intro_qpe_model} with QBCs. Let
$$
\bm{u}^*(t_m) = [\,u(x_0, t_m), \dots, u(x_N, t_m)\,]^T, \quad  \quad 1\leq m\leq M
$$
denote the exact solution vector at \(x_0, \dots, x_N\) at  \(t=t_m\), and let
$$
\bm{u}_h^m = [\,u_0^m, \dots, u_N^m\,]^T
$$
be the corresponding numerical solution. The error vector is defined by
$$
\bm{e}_h^\tau = \bm{u}^* - \bm{u}_h^M,
$$
and is measured in the  $\ell^\infty$-norm 
$$
\|\bm{e}_h^\tau\|_\infty = \max_{0\le i\le N} |e_i^\tau| = \max_{0\le i\le N} |u(x_i,T)-u_i^M|.
$$
The following theorem estimates the fully discrete error for QBC-I.

\begin{theorem}\label{th:QPEFD}
Consider the QPE \eqref{eq:intro_qpe_model} with QBC-I.
Let $p\geq 2$, $q\geq 1$, and set 
$S=\|A\|_{\infty}.
$
Assume that $\alpha\in C^p(\Omega)$ and, for each
$t\in[0,T]$,
$
f(\cdot,t),\,u(\cdot,t)\in C^p(\Omega)$, 
$u(\cdot,t)\in\mathcal H_{\rm QP}^{\rho}(\mathbb R^d),
$
where
$
\rho>k+1+n/2.
$
Assume further that
$
u\in C^{q+1}\bigl([0,T];C(\Omega)\bigr),
$
and that the spectral-separation condition
 \eqref{eq:spec_separa} holds.
Then there exists a constant $C_{S,T}>0$ such that the fully discrete
error satisfies
\[
\|\bm u^*(t_M)-\bm u_h^M\|_{\infty}
\leq
C_{S,T}
\left(
h^{\min\{p,k+1\}}+\tau^q
\right).
\]
\end{theorem}
To prove the main result in \Cref{th:QPEFD}, we first establish an error estimate for the semi-discrete scheme in \Cref{th:QPESD}. For simplicity, the proof is presented for the CDF-2 discretization with QBC-I. The corresponding estimates for higher-order CDF-\(p\) schemes can be derived analogously.

\begin{theorem}\label{th:QPESD}
Let $\bm u_h(t)$ be the semi-discrete solution obtained with CDF-2
in space and QBC-I. Under the spatial regularity assumptions and the
spectral-separation condition of Theorem~\ref{th:QPEFD}, there exists
a constant $C_{S,T}>0$, independent of $h$, such that, for all
$t\in[0,T]$,
\[
\|\bm u^*(t)-\bm u_h(t)\|_\infty
\leq
C_{S}h^{\min\{2,k+1\}}.
\]
\begin{proof}
We define the discrete operator $\mathcal{F}_h$ by
\begin{equation*}
    (\mathcal{F}_h\bm{u}_h)(t)\coloneqq \frac{d{\bm{u}}_h}{dt}(t)+A\bm{u}_h(t).
\end{equation*}
The spatial discretization, based on the CDF-2 scheme with QBC-I, leads to the semi-discrete system \eqref{eq:semiode}.

\textbf{Consistency.} 
The  CDF-2 approximation of the diffusion term satisfies
\begin{equation*}
-\frac{1}{h^2} \left( \alpha_{i+\frac{1}{2}}(u_{i+1}(t) - u_i(t)) - \alpha_{i-\frac{1}{2}}(u_i(t) - u_{i-1}(t)) \right) = -\frac{\partial}{\partial x} \left( \alpha\frac{\partial u}{\partial x} \right)(x_i,t) + \mathcal{O}(h^2).
\end{equation*}



By \Cref{lem:FPR_error1}, the boundary recovery error is $\mathcal{O}(h^{k+1})$.
Combining it with the $\mathcal{O}(h^2)$ interior error gives
\[
(\mathcal{F}_h \bm{u}^*)(t)=\bm{f}_h(t)+\bm{e}_h(t),\qquad
\|\bm{e}_h(t)\|_\infty=O(h^{\min\{2,k+1\}}),
\]
which implies that the scheme is $\min\{2,\,k+1\}$-order consistent.

\textbf{Stability.}
Consider the semi-discrete scheme
\begin{equation*}
\frac{d}{dt} \bm{u}_h(t) + A \bm{u}_h(t) = \bm{f}_h(t), \quad \bm{u}_h(0) = \bm{u}_h^0.
\end{equation*}
Integrating the semi-discrete equation over the time interval $[0, t]$ yields
\begin{equation}\label{interstab}
\bm{u}_h(t) + \int_0^t A \bm{u}_h(\eta)  d\eta = \bm{u}_h^0 + \int_0^t \bm{f}_h(\eta)  d\eta.
\end{equation}
Taking the $\ell^\infty$-norm in \eqref{interstab} and using the triangle inequality yields
\begin{equation}\label{inteineq}
\|\bm{u}_h(t)\|_\infty \le \left( \|\bm{u}_h^0\|_\infty + \int_0^t \|\bm{f}_h(\eta)\|_\infty  d\eta \right) + \|A\|_\infty \int_0^t \|\bm{u}_h(\eta)\|_\infty  d\eta.
\end{equation}
By Gronwall's inequality \cite{Jiang2025Convergence}, we have
\begin{equation*}
\|\bm{u}_h(t)\|_\infty \le e^{\|A\|_\infty t} \left( \|\bm{u}_h^0\|_\infty + \int_0^t \|\bm{f}_h(\eta)\|_\infty  d\eta \right).
\end{equation*}
Since the discrete operator $A$ is bounded,  $\|A\|_\infty$ is bounded, and thus the constant $C = e^{\|A\|_\infty T}$ is finite, which concludes the proof.


\textbf{Convergence.} Consider the error 
\begin{equation*}
    \bm{e}_h(t)=\bm{u}^*(t)-\bm{u}_h(t).
\end{equation*}
The error equation can be written as
\begin{equation*}
    (\mathcal{F}_h \bm{e}_h)(t) =\mathcal{R}(t) + \bm{b}_h(t),
\end{equation*}
where $\mathcal{R}(t)$ is the interior truncation error, and $\bm{b}_h(t)$ is the boundary recovery error introduced by the FPR method. According to the theory of the FPR method, the boundary error satisfies
\begin{equation*}
    \|\bm{b}_h(t)\|_\infty = \mathcal{O}(h^{k+1}).
\end{equation*}
By the Lax-Richtmyer equivalence theorem \cite{Lax1956Survey}, we obtain
$$
\|\bm{e}_h(t)\|_\infty \le \|\mathcal{F}_h^{-1}\| \cdot (\|\mathcal{R}(t)\|_\infty + \|\bm{b}_h(t)\|_\infty) \le C_1 h^2 + C_2 h^{k+1}.
$$
Hence, the spatial discretization error is estimated by
$$
\| \bm{u}^*(t)-\bm{u}_h(t) \|_\infty = \mathcal{O}(h^2) + \mathcal{O}(h^{k+1}) = \mathcal{O}(h^{\min\{2,\, k+1\}}).
$$    
\end{proof}

\begin{remark}
The above analysis  can be directly extended  to  high-order CDF schemes. 
Specifically, suppose that a $p$-th order CDF discretization is applied in the interior, 
while the boundary is treated by the QBC-I(FPR$k$) approach. Then the numerical solution satisfies
\[
\|\bm{u}^*(t) - \bm{u}_h(t)\|_\infty \le C_S \big(h^{\min\{p,k+1\}}\big).
\]
\end{remark}


\begin{remark}
The preceding results consider the QPE~\eqref{eq:intro_qpe_model} with
high-regularity coefficients. We now turn to the nonsmooth-coefficient case.
For such problems, the standard central difference approximation may lose accuracy,
so we employ a flux-difference discretization. In the 1D case,
\[
-\frac{d}{dx}\left(\alpha(x)\frac{du}{dx}\right)
\approx
-\frac{1}{h}\left(F_{i+\frac{1}{2}}-F_{i-\frac{1}{2}}\right),
\qquad
F_{i+\frac{1}{2}}
=
\alpha_{i+\frac{1}{2}}
\frac{u_{i+1}(t)-u_i(t)}{h}.
\]
The interface value $\alpha_{i+\frac{1}{2}}$ is computed by harmonic averaging
to ensure numerical stability and consistency~\cite{ewing2001modified}.
The resulting discretization retains second-order accuracy.
\end{remark}

\end{theorem}

For the temporal error analysis, by
standard error estimates for BDF methods (see, e.g. \cite{hairer1996solving}), we have the following result.

\begin{theorem}\label{th:QPETD}
Let $\bm{u}_h(t)$ be the solution of the semi-discrete problem \eqref{eq:semiode} for $t \in [0,T]$, with
$
\bm{u}_h \in C^{q+1}([0,T]).
$
Let $\bm{u}_h^M$ denote the fully discrete solution at time $t_M = T$ obtained by the BDF-$q$ scheme with time step $\tau$. Then there exists a constant $C_T > 0$, independent of $\tau$, such that
\begin{equation}
\|\bm{u}_h(t_M) - \bm{u}_h^M\|_\infty \le C_T \, \tau^q.
\end{equation}
\end{theorem}

\textbf{Proof of Theorem 4.1.} Using  \Cref{th:QPESD} and \Cref{th:QPETD}, we are now in a position to provide the error estimate for the fully discrete scheme \eqref{eq:bdfk_general}.
\begin{proof}
By the triangle inequality, we have
$$
\|\bm{u}^*- \bm{u}_h^M\|_\infty \le \|\bm{u}(t_M) - \bm{u}_h(t_M)\|_\infty + \|\bm{u}_h(t_M) - \bm{u}_h^M\|_\infty.
$$
By \Cref{th:QPESD}, we obtain
$$
\|\bm{u}(t_M) - \bm{u}_h(t_M)\|_\infty \le C_S\, h^{\min\{p,k+1\}},
$$
and   \Cref{th:QPETD} yields
$$
\|\bm{u}_h(t_M) - \bm{u}_h^M\|_\infty \le C_T \, \tau^q.
$$
Combining these results gives the fully discrete error estimate
$$
\|\bm{u}^* - \bm{u}_h^M\|_\infty \le C_{S,T} \big(h^{\min\{p,k+1\}} + \tau^q \big),
$$
where $C_{S,T}$ depends on $S=\|A\|_\infty$ and the final time $T$. 
\end{proof}
\begin{remark}
For the QPE \eqref{eq:intro_qpe_model} with QBC-II, the errors of the CDF-$p$
and BDF-$q$ discretizations are $\mathcal{O}(h^p)$ and $\mathcal{O}(\tau^q)$,
respectively. Moreover, combining \eqref{eq:qbc2_fpr} and \eqref{eq:qbc2_cdf} 
and applying the triangle inequality, we obtain the following estimate for the boundary derivative approximation error on \( \partial\Omega\)
\[
\begin{aligned}
\left\|
\frac{\partial u}{\partial x}
-
\left(
\mathcal{T}_{\mathrm{FPR}}
\frac{\partial u}{\partial x}
\right)_h
\right\|_{\infty}
\leq{}&
\left\|
\frac{\partial u}{\partial x}
-
\mathcal{T}_{\mathrm{FPR}}
\frac{\partial u}{\partial x}
\right\|_{\infty}
+
\left\|
\mathcal{T}_{\mathrm{FPR}}
\frac{\partial u}{\partial x}
-
\left(
\mathcal{T}_{\mathrm{FPR}}
\frac{\partial u}{\partial x}
\right)_h
\right\|_{\infty}
\\
={}&
\left\|
\frac{\partial u}{\partial x}
-
\sum_{x_i\in\Omega_N}
\left.\frac{\partial u}{\partial x}\right|_{x=x_i}
\phi_i(\Phi_x^{\bm P})
\right\|_{\infty}
\\&+
\left\|
\sum_{x_i\in\Omega_N}
\left[
\left.\frac{\partial u}{\partial x}\right|_{x=x_i}
-
\frac{u_{i+1}-u_{i-1}}{2h}
\right]
\phi_i(\Phi_x^{\bm P})
\right\|_{\infty}
\\
\leq{}&
C_1h^{k+1}+C_2h^2
\leq
Ch^{\min\{k+1,2\}} .
\end{aligned}
\]
where $\left(\cdot
\right)_h$ denotes the form in which the derivative is discretized by a CDF scheme on the spatial mesh of size \(h\). This estimate extends directly to a CDF-$l$ approximation, for which
\[
\left\|
\frac{\partial u}{\partial x}
-
\left(
\mathcal{T}_{\mathrm{FPR}}
\frac{\partial u}{\partial x}
\right)_h
\right\|_{\infty}
\leq
C_1h^{k+1}+C_2h^l
\leq
Ch^{\min\{k+1,l\}} .
\]
Since QBC-III involves a linear combination of the function value and its normal derivative, its boundary approximation error is also 
$\mathcal{O}\!\left(h^{\min\{k+1,l\}}\right)$.
Therefore, for both QBC-II and QBC-III,
\[
\|\bm u^*-\bm u_h^M\|_\infty
\leq
C_{S,T}
\left(
h^{\min\{p,k+1,l\}}+\tau^q
\right).
\]
\end{remark}
\section{Numerical experiments}\label{sec:num}
In this section, we solve the QPE~\eqref{eq:intro_qpe_model} using the CDF scheme in space and the BDF method in time, with QBCs imposed on the boundary. To demonstrate the wide suitability and high accuracy of the proposed method, we conduct extensive numerical experiments for both high- and low-regularity cases. Since QBCs are imposed only in the spatial direction, the tests focus on spatial convergence.
The temporal convergence orders of the BDF-$q$ schemes with $q=2$ and $4$ have been verified separately.
All experiments are carried out in MATLAB R2021b on a personal computer equipped with an AMD Ryzen 7 CPU and 16 GB of RAM.



To compute the spatial convergence order, we fix  $\tau$
and refine the spatial mesh. Denote by $\|\bm{e}_{h}^\tau\|_\infty$
the error at the final time.  
The spatial  order is given by
\[
\kappa = \frac{\ln \big(\|\bm{e}_{h_1}^\tau\|_\infty / \|\bm{e}_{h_2}^\tau\|_\infty \big)}{\ln(h_1/h_2)}.
\]

\subsection{High-regularity cases}
In this subsection, various numerical examples are presented to examine the convergence behavior of the numerical method with the proposed QBCs for solving the QPE \eqref{eq:intro_qpe_model} with a smooth coefficient. 
\subsubsection{One-dimensional cases}\label{subsubsec:1dcase}
We first consider the 1D QPE \eqref{eq:intro_qpe_model} to investigate the numerical performance of the proposed method. The quasiperiodic coefficient and the exact solution are given by
\begin{equation}\label{eq:ustar_1d}
  \alpha(x)
=6+\cos x+\sum_{i=1}^{s}\cos(\beta_i x),\qquad   u^*(x,t)
=e^{-t}\left(\sin x+\sum_{i=1}^{s}\sin(\beta_i x)\right),
\end{equation}
where $x\in\mathbb{R}$  and $\beta_i\in\mathbb{R}\setminus\mathbb{Q}$. 
The computational domain is \(\Omega=[-L\pi,L\pi]\), and the final time is \(T=2.5\). The spatial and temporal discretizations are based on CDF-2 and BDF-2, respectively.  
To examine the spatial convergence, we fix $\tau=1/512$ and refine the mesh size as $h=\pi/(10\cdot 2^j),~ j = 1, 2, 3, 4$.
 Spatial errors and convergence rates are reported, with \(L=700\) used for QBC-I/II and \(L=500\) for QBC-III. 
As shown in \Cref{fig:qbc_all_ell_infty}, the FPR-$k$ schemes $(k=1,2,3)$ with QBC-I, II, and III achieve second-order convergence in space, 
confirming the expected accuracy of the spatial discretization.

\begin{figure}[htbp]
    \centering
    \setlength{\tabcolsep}{2pt}
    \begin{tabular}{ccc}
        \includegraphics[width=0.30\textwidth]{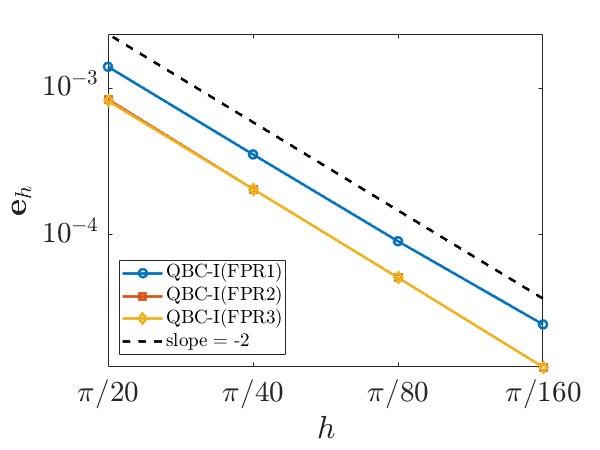} &
        \includegraphics[width=0.3\textwidth]{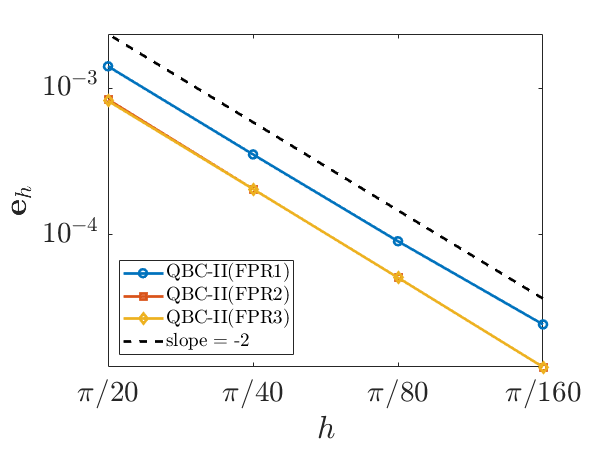} &
        \includegraphics[width=0.3\textwidth]{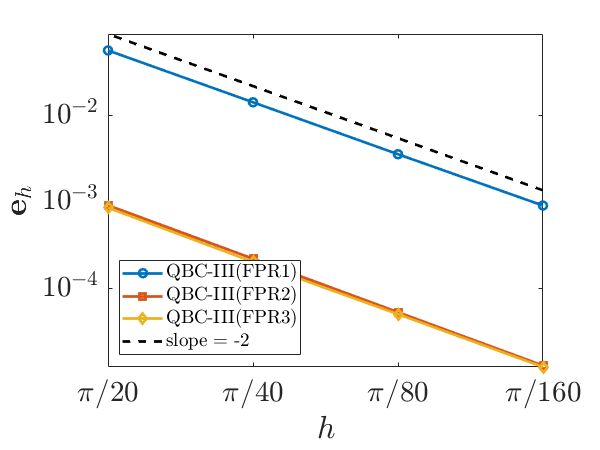}
    \end{tabular}
  \caption{
Second-order spatial convergence of QBC-I, II, and III for the 1D QPE~\eqref{eq:intro_qpe_model} with $p=2$ and $\beta_1=(\sqrt{5}-1)/2$.
}
    \label{fig:qbc_all_ell_infty}
\end{figure}

The proposed QBCs are also applicable to higher-order discretizations. To illustrate this flexibility, we consider the 1D QPE~\eqref{eq:ustar_1d} using QBC-I together with CDF-4 in space and BDF-4 in time.  For this and the
subsequent examples, QBC-I is used as a representative case, with similar
results for QBC-II and QBC-III.
The numerical results in \Cref{fig:qbc_I_ho_ell_infty} are consistent with
the error estimate in \Cref{th:QPEFD}, which gives the
spatial convergence order $\min\{4,k+1\}$ for $p=4$.
Accordingly, QBC-I(FPR1) and QBC-I(FPR2) achieve approximately second- and
third-order spatial convergence, respectively, while QBC-I(FPR$k$) with
$k\geq 3$ attains fourth-order convergence.
These results agree with the preceding numerical analysis and confirm the compatibility of QBCs with higher-order discretizations.

We next consider a more general case involving multiple irrational frequencies. Specifically, the QPE~\eqref{eq:intro_qpe_model} is solved with the quasiperiodic coefficient and exact solution involving $ \beta_1=\pi$, $\beta_2=\sqrt{3}$. The spatial and temporal discretizations are based on CDF-2 and BDF-2, respectively, and QBC-I is imposed at the boundary with $L=6000$. As shown in \Cref{fig:qbc_I_ell_infty_1to3}, the QBC-I(FPR$k$) schemes $(k=1,2,3)$ achieve second-order convergence in space. 

These results show that the proposed QBCs remain effective for QPEs involving multiple irrational numbers.

\begin{figure}[!htbp]
    \centering
    \subfigure[]{
        \includegraphics[width=4.5cm]{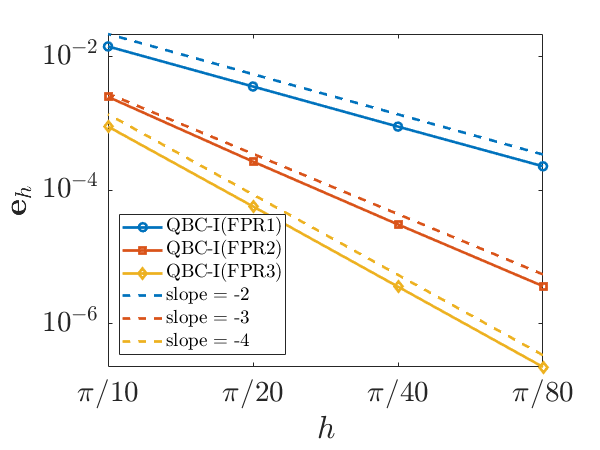}
        \label{fig:qbc_I_ho_ell_infty}
    }
    \hspace{2mm}
    \subfigure[]{
        \includegraphics[width=4.5cm]{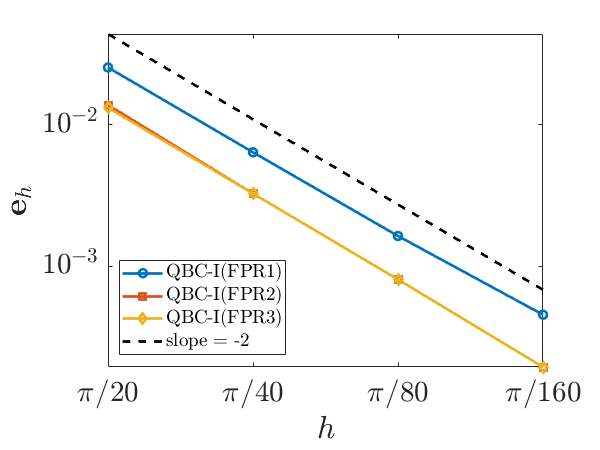}
        \label{fig:qbc_I_ell_infty_1to3}
    }
\caption{
Spatial convergence results of QBC-I for the 1D QPE~\eqref{eq:intro_qpe_model}, 
(a) $p=4$, $\beta_1=(\sqrt{5}-1)/2$, and $L=2000$;
(b) $p=2$, $\beta_1=\pi$, $\beta_2=\sqrt{3}$, and $L=6000$.
}
\label{fig:qbc_I_ell_infty_fororder_1to3}
\end{figure}

\textbf{QBC vs. PBC.}
~In this part, we provide a comprehensive comparison between QBC and PBC for solving the 1D finite-size QPE~\eqref{eq:intro_qpe_model}. 
\Cref{app:PBC} presents the numerical implementation of PBCs.

\textbf{Case 1:   $\beta_1=\sqrt{2}$.}
\begin{figure}[!htbp]
\centering
    \includegraphics[width=0.6\linewidth]{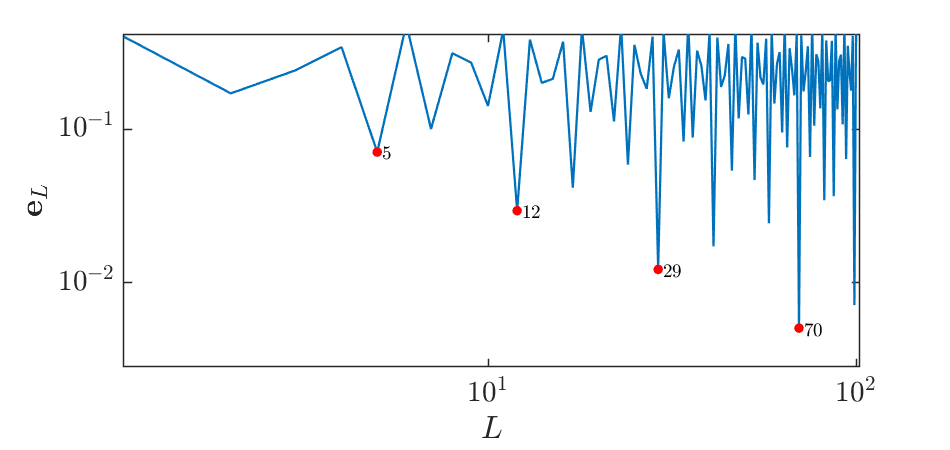}
\caption{Diophantine errors  of $\sqrt{2}$. }
\label{fig:pbc_sr5_2}
\end{figure}
We first compare PBC and QBC-I for the 1D QPE~\eqref{eq:intro_qpe_model} involving the single irrational frequency $\beta_1=\sqrt{2}$. To make the spatial discretization error negligible
in the comparison between PBC and QBC-I, we take a sufficiently small mesh
size $h=\pi/10240$ and present the corresponding errors in \Cref{tab:error_pbc_qbc_p24}. The Diophantine error of  $\sqrt{2}$  is shown in   \Cref{fig:pbc_sr5_2}, where the red points indicate the convergent denominators. 

For PBC, the error does not decrease monotonically with increasing domain size $L$ and decreases consistently only when $L$ is chosen from the convergent denominators, including the bold values $L=12$ and $L=70$.
For example, the error at $L=43$ is larger than that at $L=12$. 
For QBC-I, the error decreases as $L$ increases. 
Moreover, higher-order discretizations further reduce the error. Specifically, the QBC-I errors decrease from $4.21\times10^{-3}$ at $L=12$ to $2.16\times10^{-5}$ at $L=70$ with CDF-2, and from $2.00\times10^{-4}$ to $3.09\times10^{-8}$ over the same range of $L$ with CDF-4.
The PBC errors, however, remain unchanged for CDF-2 and CDF-4. 
Notably, the smallest PBC error is $1.58\times10^{-2}$, whereas
QBC-I(FPR1) achieves a smaller error of $4.21\times10^{-3}$ using only
about $17.1\%$ of the domain size required to obtain the best PBC result.
These results show that QBC-I avoids the influence of the Diophantine error.

\begin{table}[htbp]
\centering
\caption{
Error comparison between PBC and QBC-I for the 1D QPE~\eqref{eq:intro_qpe_model}
with $\beta_1=\sqrt{2}$, where $h=\pi/10240$.
}

\label{tab:error_pbc_qbc_p24}

\small
\renewcommand{\arraystretch}{1}
\setlength{\tabcolsep}{6pt}

\begin{tabular}{cccccc}
\toprule
& &
\multicolumn{2}{c}{$p=2$}
&
\multicolumn{2}{c}{$p=4$}
\\
\cmidrule(lr){3-4}
\cmidrule(lr){5-6}

$L$ & $\bm{e}_L$
& PBC
& QBC-I(FPR1)
& PBC
& QBC-I(FPR3)
\\
\midrule

$\bm{12}$
& $2.94\mathrm{e}{-02}$
& $9.20\mathrm{e}{-02}$
& $4.21\mathrm{e}{-03}$
& $9.20\mathrm{e}{-02}$
& $2.00\mathrm{e}{-04}$
\\

$31$
& $1.59\mathrm{e}{-01}$
& $4.78\mathrm{e}{-01}$
& $3.12\mathrm{e}{-03}$
& $4.78\mathrm{e}{-01}$
& $2.47\mathrm{e}{-05}$
\\

$43$
& $1.89\mathrm{e}{-01}$
& $5.56\mathrm{e}{-01}$
& $2.24\mathrm{e}{-03}$
& $5.56\mathrm{e}{-01}$
& $9.27\mathrm{e}{-06}$
\\

$\bm{70}$
& $5.05\mathrm{e}{-03}$
& $1.58\mathrm{e}{-02}$
& $2.16\mathrm{e}{-05}$
& $1.58\mathrm{e}{-02}$
& $3.09\mathrm{e}{-08}$
\\

\bottomrule
\end{tabular}
\end{table}

\textbf{Case 2: $\beta_1= (\sqrt{5}-1)/2$, $\beta_2=\sqrt{2}$.} We next extend the comparison to a 1D QPE~\eqref{eq:intro_qpe_model} with two irrational frequencies $(\sqrt{5}-1)/{2}, \sqrt{2}.$  The Diophantine error is shown in \Cref{fig:pbc_tw_ir_num_gold_s2}, where the highlighted red points indicate 
the common denominators providing small simultaneous Diophantine errors for the two irrational frequencies $\beta_1$ and $\beta_2$. 
The error comparison between PBC and QBC-I is reported in \Cref{tab:error_pbc_qbc_p24_gold_s2}.  
\begin{figure}[!htbp]
\centering
    \includegraphics[width=0.6\linewidth]{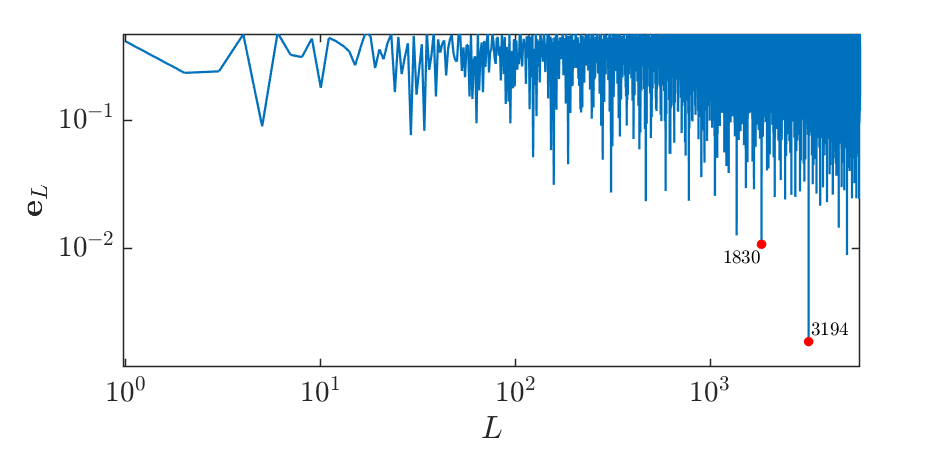}
\caption{Diophantine errors  of $(\sqrt{5}-1)/2$, $\sqrt{2}$. }
\label{fig:pbc_tw_ir_num_gold_s2}
\end{figure}
Similar to the single-irrational-frequency case, the PBC error is affected by the Diophantine error and does not decrease monotonically as the domain size $L$ increases. 
\begin{table}[htbp]
\centering
\caption{
Error comparison between PBC and QBC-I for the 1D QPE~\eqref{eq:intro_qpe_model}
with $\beta_1=(\sqrt{5}-1)/2$, $\beta_2=\sqrt{2}$, where $h=\pi/10240$.
}
 \label{tab:error_pbc_qbc_p24_gold_s2}

\small
\renewcommand{\arraystretch}{1}
\setlength{\tabcolsep}{6pt}

\begin{tabular}{cccccc}
\toprule
& &
\multicolumn{2}{c}{$p=2$}
&
\multicolumn{2}{c}{$p=4$}
\\
\cmidrule(lr){3-4}
\cmidrule(lr){5-6}

$L$ & $\bm{e}_L$
& PBC
& QBC-I(FPR1)
& PBC
& QBC-I(FPR3)
\\
\midrule

$1601$
& $4.72\mathrm{e}{-01}$
& $5.25\mathrm{e}{-01}$
& $2.68\mathrm{e}{-03}$
& $5.25\mathrm{e}{-01}$
& $2.32\mathrm{e}{-05}$
\\

$\bm{1830}$
& $1.08\mathrm{e}{-02}$
& $2.69\mathrm{e}{-02}$
& $5.50\mathrm{e}{-05}$
& $2.69\mathrm{e}{-02}$
& $1.90\mathrm{e}{-07}$
\\

$\bm{3194}$
& $1.88\mathrm{e}{-03}$
& $7.66\mathrm{e}{-03}$
& $1.90\mathrm{e}{-05}$
& $7.66\mathrm{e}{-03}$
& $9.62\mathrm{e}{-08}$
\\

$5176$
& $5.61\mathrm{e}{-02}$
& $7.92\mathrm{e}{-02}$
& $6.54\mathrm{e}{-06}$
& $7.92\mathrm{e}{-02}$
& $1.96\mathrm{e}{-08}$
\\

\bottomrule
\end{tabular}
\end{table}
The smallest PBC error is
$7.66\times10^{-3}$ at the selected common denominator $L=3194$.
In contrast, QBC-I(FPR1) already achieves a smaller error of
$2.68\times10^{-3}$ at $L=1601$, which is about $50.1\%$ of the domain
size required by PBC. These results further demonstrate that QBC-I
avoids the influence of the Diophantine error and achieves
higher accuracy on a smaller computational domain than PBC.

\textbf{Case 3: $\beta_1=\pi$, $\beta_2=\sqrt{3}$.} 
The previous two-irrational-number example involves irrational numbers with similar approximation behavior. We now consider a pair with different Diophantine approximation properties, namely $\pi$ and $\sqrt{3}$. This setting is more challenging for PBC, since the corresponding Diophantine error becomes small only for much larger values of $L$. Therefore, the periods $L=2034$ and $L=12995$, selected from the simultaneous Diophantine approximation shown in  \Cref{fig:pbc_tw_ir_num_s3pi}, are used in the error comparison. 
\begin{figure}[!htbp]
\centering
    \includegraphics[width=0.6\linewidth]{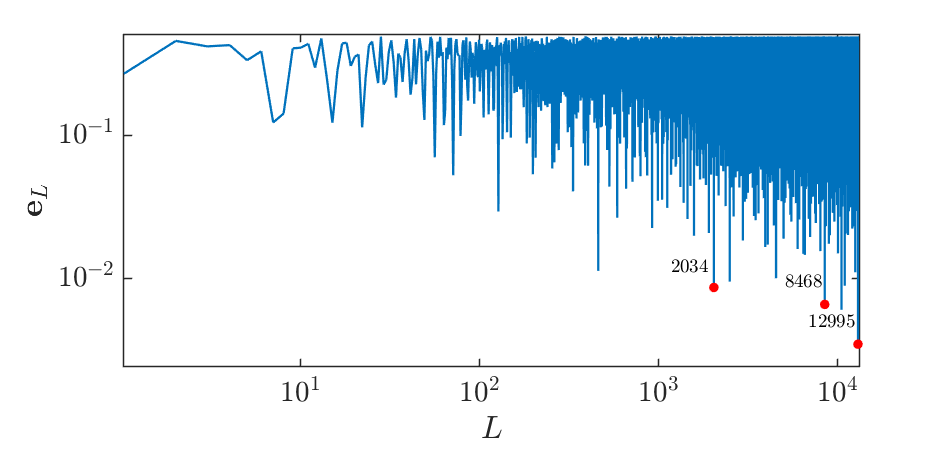}
\caption{Diophantine errors  of $\sqrt{3}$, $\pi$. }
\label{fig:pbc_tw_ir_num_s3pi}
\end{figure}

Similarly, \Cref{tab:error_pbc_qbc_p24_pi_s3} shows that QBC-I avoids the influence of the Diophantine error.
The smallest PBC error is $1.16\times10^{-2}$ at $L=12995$. In contrast, QBC-I(FPR1) already achieves a smaller error
of $1.53\times10^{-3}$ at $L=1800$, which is only about $13.9\%$ of the
domain size required by PBC. 
These results confirm that QBCs remain effective for irrational frequencies with different approximation properties.

\begin{table}[htbp]
\centering
\caption{
Error comparison between PBC and QBC-I for the 1D QPE~\eqref{eq:intro_qpe_model} with $\beta_1=\pi$, $\beta_2=\sqrt{3}$, where $h=\pi/10240$.
}
\label{tab:error_pbc_qbc_p24_pi_s3}

\small
\renewcommand{\arraystretch}{1}
\setlength{\tabcolsep}{6pt}

\begin{tabular}{cccccc}
\toprule
& &
\multicolumn{2}{c}{$p=2$} &
\multicolumn{2}{c}{$p=4$} \\
\cmidrule(lr){3-4}
\cmidrule(lr){5-6}

$L$ & $\bm{e}_L$
& PBC & QBC-I(FPR1)
& PBC & QBC-I(FPR3) \\
\midrule

$1800$
& $3.09\mathrm{e}{-01}$
& $4.17\mathrm{e}{-01}$
& $1.53\mathrm{e}{-03}$
& $4.17\mathrm{e}{-01}$
& $9.80\mathrm{e}{-06}$ \\

$\bm{2034}$
& $8.66\mathrm{e}{-03}$
& $2.54\mathrm{e}{-02}$
& $7.99\mathrm{e}{-04}$
& $2.28\mathrm{e}{-02}$
& $5.68\mathrm{e}{-06}$ \\

$8493$
& $4.54\mathrm{e}{-01}$
& $1.66\mathrm{e}{-01}$
& $2.32\mathrm{e}{-04}$
& $1.66\mathrm{e}{-01}$
& $1.62\mathrm{e}{-08}$ \\

$\bm{12995}$
& $3.47\mathrm{e}{-03}$
& $1.16\mathrm{e}{-02}$
& $3.28\mathrm{e}{-06}$
& $1.16\mathrm{e}{-02}$
& $2.05\mathrm{e}{-09}$ \\

\bottomrule
\end{tabular}
\end{table}

\subsubsection{Two-dimensional case} Having examined the one-dimensional case, we next consider high-dimensional QPEs~\eqref{eq:intro_qpe_model} to further demonstrate that the proposed QBCs remain effective in higher dimensions.
We consider the quasiperiodic coefficient
\begin{equation*}
    \alpha(\bx)
    =6 +\sum_{i=1}^r\cos \bm{\mu}_i\cdot\bx+\sum_{i=1}^{s}\cos \bm{\beta}_i\cdot\bx , \qquad\bx\in\mathbb{R}^2.
\end{equation*}
The exact solution is
\begin{equation*}
u^*(\bx,t)
= e^{-t}\left(\sum_{i=1}^r\sin \bm{\mu}_i\cdot\bx+\sum_{i=1}^{s}\sin \bm{\beta}_i\cdot \bx \right), \qquad\bx\in\mathbb{R}^2,
\end{equation*}
where $\bm{\beta}_i$ contains irrational components. 
In the computation, we take $\bm {\mu}_1=(1,0)^T$, 
$\bm {\mu}_2=(0,1)^T,$ 
$\bm{\beta}_1=(\sqrt{2},0)^T$. 
The computational domain is  
$
\Omega = [-100\pi,100\pi] \times [-100\pi,100\pi].
$
We employ the CDF-2 scheme for spatial discretization and the BDF-2 scheme for temporal discretization. 
The results in \Cref{fig:qbc_I_ell_infty_2to3} show that QBC achieves second-order convergence in space for QBC-I,  further confirming its effectiveness for higher-dimensional QPEs.

\begin{figure}[htbp]
\centering
\includegraphics[width=0.4\linewidth]{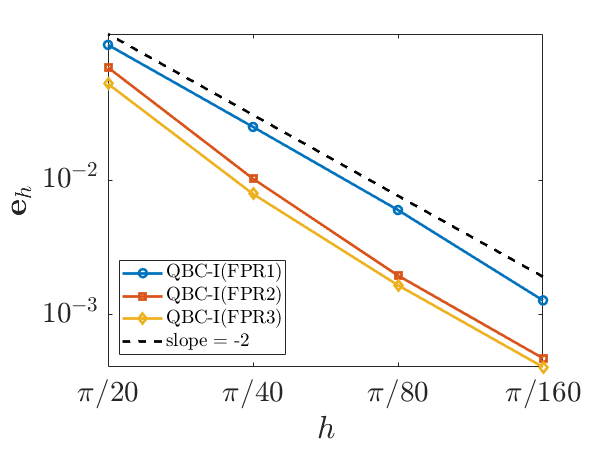}
\caption{
Second-order spatial convergence of QBC-I for the 2D QPE~\eqref{eq:intro_qpe_model} with $p=2$ and $\bm{\beta}_1=(\sqrt{2},0)^T$.}
\label{fig:qbc_I_ell_infty_2to3}
\end{figure}
\subsection{Low-regularity cases} 

We next test the proposed QBCs for low-regularity quasiperiodic coefficients. Two examples are considered. The first one is the QPE~\eqref{eq:intro_qpe_model} with the following $C^0$ quasiperiodic coefficient
\begin{equation}
\label{eq:c0_quasi_coeffi}
\alpha(x)
= 6+|\cos x|+\sum_{i=1}^{s}\left|\cos \beta_i x\right|,\qquad x\in \mathbb{R},~ \beta_i\in\mathbb{R}\setminus\mathbb{Q}.
\end{equation}
The exact solution is chosen in the form of \eqref{eq:ustar_1d}. In this test, we take $\beta_1=\sqrt{2}$, and set the computational domain to $\Omega=[-700\pi,700\pi]$. 
The second example considers a Fibonacci quasiperiodic coefficient
$\alpha(x)$, as shown in \Cref{fig:fbnq}.
The computational domain is chosen as $\Omega=[\pi,100\pi]$. In both tests, the computation is performed using the second-order discretization scheme together with QBC-I.

The spatial convergence results for the $C^0$ coefficient
\eqref{eq:c0_quasi_coeffi} are reported in
\Cref{fig:qbc_I_ell_infty_c0}. QBC-I(FPR$k$) $(k=1,2,3)$
achieves second-order convergence in space. 
The spatial convergence results for the Fibonacci coefficient are shown in
\Cref{fig:qbc_I_ell_infty_fbnq}. We present only QBC-I(FPR1) because the
piecewise-constant structure limits the model regularity, and the resulting
error is dominated by this low regularity. Nevertheless, QBC-I(FPR1)
attains the expected second-order convergence. These results indicate that
the proposed QBCs remain effective for low-regularity quasiperiodic
coefficients.
\begin{figure}[!htbp]
\centering

\subfigure[]{
    \includegraphics[width=0.30\linewidth]{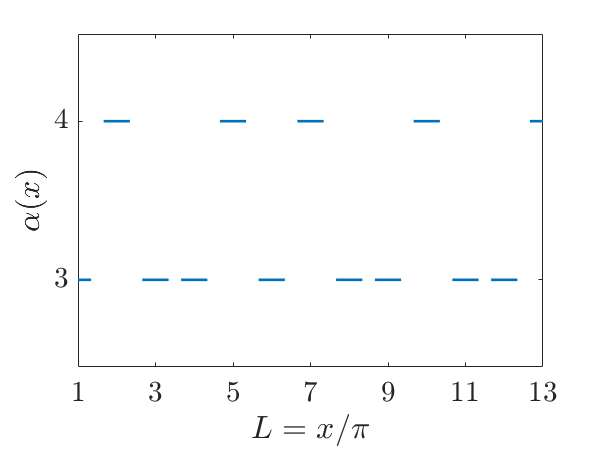}
    \label{fig:fbnq}
}
\hfill
\subfigure[]{
    \includegraphics[width=0.30\linewidth]{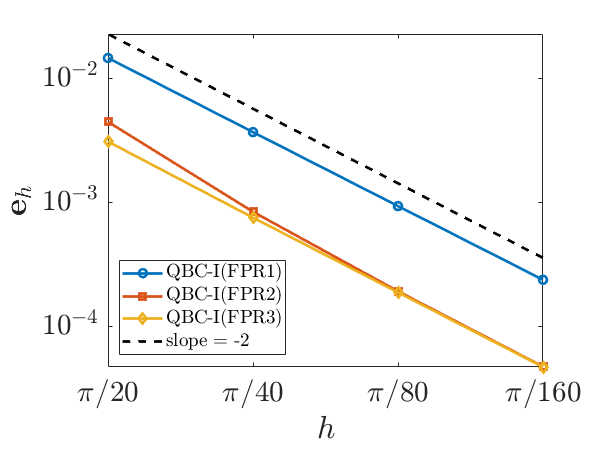}
    \label{fig:qbc_I_ell_infty_c0}
}
\hfill
\subfigure[]{
    \includegraphics[width=0.30\linewidth]{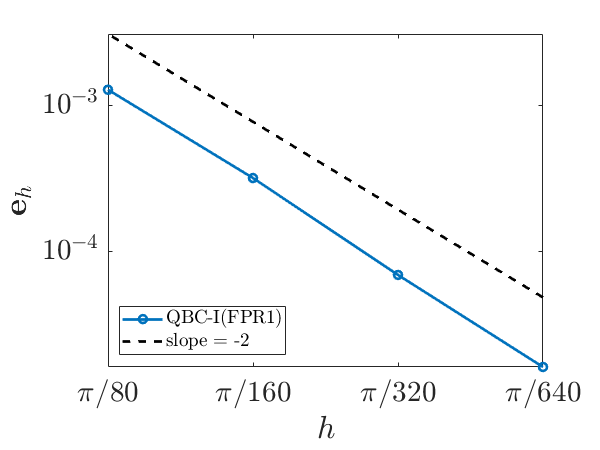}
    \label{fig:qbc_I_ell_infty_fbnq}
}

\caption{
Low-regularity cases for the 1D QPE~\eqref{eq:intro_qpe_model}.
(a) The Fibonacci coefficient $\alpha(x)$ on $x\in[\pi,13\pi]$.
(b) Spatial convergence for the $C^0$ coefficient~\eqref{eq:c0_quasi_coeffi}.
(c) Spatial convergence for the Fibonacci coefficient shown in \Cref{fig:fbnq}.
}
\label{fig:lowreg}
\end{figure}

\textbf{QBC \textbf{vs.} PM.}~~ 
To further evaluate the accuracy and efficiency of QBCs,  we compare QBC-I with the PM in solving low-regularity QPEs \eqref{eq:intro_qpe_model} with  the $C^0$ coefficient~\eqref{eq:c0_quasi_coeffi} 
and the Fibonacci coefficient shown in \Cref{fig:fbnq}. The final time and time step are fixed as $T=10^{-4}$ and $\tau=10^{-6}$, respectively. 
For QBC-I, we use $\Omega=[-210\pi,210\pi]$ and $[\pi,100\pi]$
for the $C^0$ and Fibonacci cases, respectively.

\begin{figure}[!htbp]
    \centering
    \subfigure[]{
        \includegraphics[width=5cm]{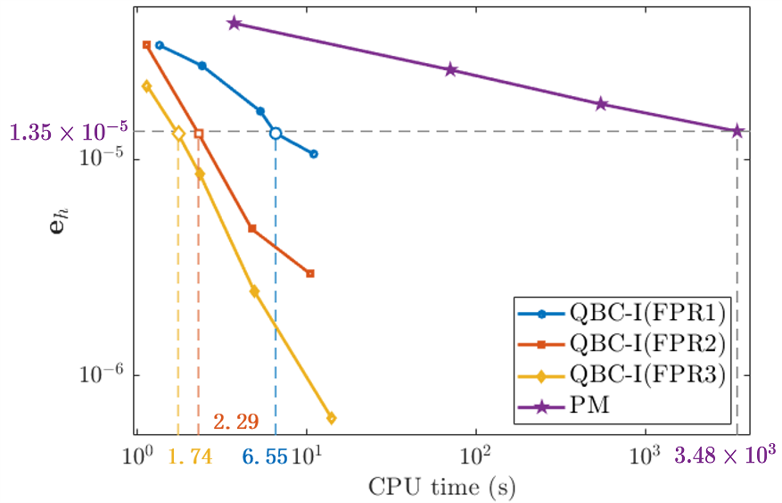}
        \label{fig:pmvsqbc_c0}
    }
    \hspace{2mm}
    \subfigure[]{
        \includegraphics[width=5cm]{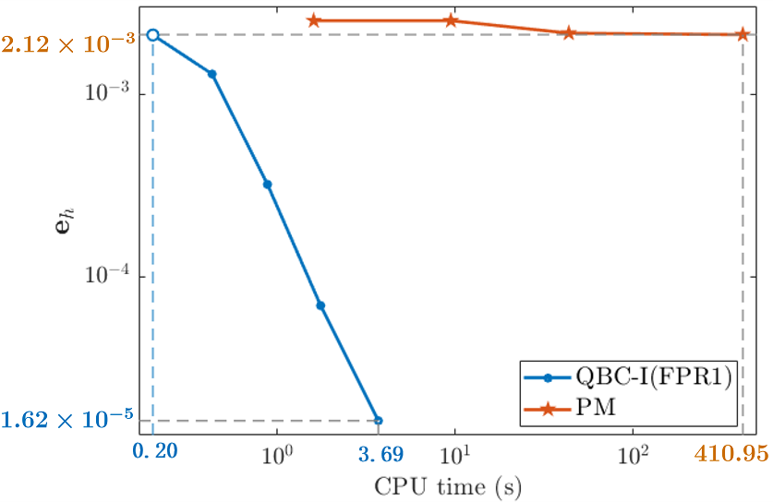}
        \label{fig:pmvsqbc_fbnq}
    }
    \caption{
Cost--accuracy trade-off between QBC-I and PM for the 1D QPE~\eqref{eq:intro_qpe_model} with $p=2$ and low-regularity
quasiperiodic coefficients,
(a) the $C^0$ coefficient~\eqref{eq:c0_quasi_coeffi};
(b) the Fibonacci coefficient.}
\label{fig:pmvsqbc_c0_and_fbnq}
\end{figure}
We first compare the computational complexity of PM and QBC-I. For the PM, the computational complexity is $\mathcal{O}(D\log D)$, where $D=N^2$ denotes the degrees of freedom and $N$ is the number of grid points in each direction. 
Meanwhile, for QBC-I(FPR$k$), the computational complexity scales as $3N+(k+1)^2\sim \mathcal{O}(N)$, since $k\ll N$. 
The accuracy-efficiency comparison between QBC-I and the PM is then presented in \Cref{fig:pmvsqbc_c0_and_fbnq}.
For the $C^0$ coefficient, the results of QBC-I(FPR$k$) ($k=1,2,3$) are shown in \Cref{fig:pmvsqbc_c0}. 
Under the same computational conditions, the PM with $N=128$ reaches an accuracy level of approximately $1.35\times 10^{-5}$, with a CPU time of approximately $3.48\times10^3$s. 
To achieve the same accuracy level, QBC-I(FPR1) requires only $6.55$s. 
The CPU time is further reduced to $2.29$s for QBC-I(FPR2), while QBC-I(FPR3) attains this accuracy in only $1.74$s. More importantly, QBC-I(FPR3) can further reduce the error below $10^{-6}$ within $10.00$s, outperforming the PM in both accuracy and efficiency.
For the Fibonacci coefficient, this advantage becomes more pronounced, since this coefficient is piecewise constant.  As shown in \Cref{fig:pmvsqbc_fbnq}, the PM requires about $410.95$s to reach the accuracy level of $10^{-3}$, whereas QBC-I(FPR1) requires only $0.20$s to achieve the same accuracy level. Meanwhile, QBC-I(FPR1) reduces the error to $1.62\times 10^{-5}$ in only $3.69$s, indicating that QBC-I can still achieve high accuracy at a very small computational cost in this case. Overall, QBC-I exhibits a significant accuracy-efficiency advantage for low-regularity QPEs. 

\section{Conclusions and outlooks}\label{sec:conc}
In this paper, we propose QBCs for finite-size computation of QPEs.
Specifically, we construct three types of QBCs based on the inherent homomorphism
between $\mathbb{R}^d$ and the irrational manifold $(\bm{P}^T\mathbb{R}^d)/\mathbb{Z}^n$.
The QBCs preserve the quasiperiodic structure of the full-space QPE solution within finite-size domains.
For the numerical implementation of QBCs, we employ the FPR method to recover
the required boundary information from interior samples.
We further establish an error analysis for the fully discrete schemes.
Numerical experiments for both high- and low-regularity cases confirm the
accuracy and efficiency of the proposed QBCs.
Compared with PBCs, QBCs avoid the influence of the Diophantine
error.
Moreover, in low-regularity cases, the numerical method with QBCs outperforms
the PM in both accuracy and efficiency.
Future study will extend the proposed QBCs beyond QPEs to a broader class of
quasiperiodic PDEs and further develop their theoretical framework. Furthermore,
we will improve the numerical implementation of QBCs using more general
discretizations and FPR interpolation. More importantly, we will explore QBCs
in complex quasiperiodic systems to uncover new physical phenomena.

\appendix
\section{Discretization of QPEs with PBCs}
\label{app:PBC} 
For the one-dimensional case of the QPE~\eqref{eq:intro_qpe_model},
consider $(x,t)\in(a,b)\times(0,T)$. The PBC is given by
\[
    u(a,\cdot)=u(b,\cdot).
\]
The corresponding discrete periodic condition is
$
    u_0(t)=u_N(t).
$
 
Consider the CDF-2 spatial discretization as a representative case, with
the same construction extending directly to general CDF-$p$ schemes.
Combining the CDF-2 discretization with the PBC gives the semi-discrete system
\[
    \frac{\mathrm d\bm{u}_h(t)}{\mathrm dt}
    +
    A\bm{u}_h(t)
    =
    \bm{f}_h(t),
\]
where
$
    A=-h^{-2}H,
$
and
\[
\scalebox{0.72}{$
H =
\begin{bmatrix}
-(\alpha_{N-1/2}+\alpha_{1/2})
& \alpha_{1/2}
& 0
& \cdots
& 0
& \alpha_{N-1/2}
& 0
\\
\alpha_{1/2}
& -(\alpha_{1/2}+\alpha_{3/2})
& \alpha_{3/2}
& \ddots
& \vdots
& 0
& 0
\\
0
& \ddots
& \ddots
& \ddots
& \ddots
& \vdots
& \vdots
\\
0
& \cdots
& 0
& \alpha_{N-3/2}
& -(\alpha_{N-3/2}+\alpha_{N-1/2})
& \alpha_{N-1/2}
& 0
\\
0
& \alpha_{1/2}
& 0
& \cdots
& 0
& \alpha_{N-1/2}
& -(\alpha_{N-1/2}+\alpha_{1/2})
\end{bmatrix}.
$}
\]
Applying the BDF-$q$ scheme to the semi-discrete system gives the fully discrete scheme
\[
    \frac{1}{\tau}
    \sum_{j=0}^{q}
    c_j\bm{u}_h^{m+1-j}
    +
    A\bm{u}_h^{m+1}
    =
    \bm{f}_h^{m+1}.
\]

\end{document}